\documentclass[11pt,amsfonts, epsfig, a4paper]{amsart}
\usepackage{amsmath, amscd, amssymb}

\usepackage[usenames, dvipsnames]{color}

\usepackage{lipsum}
\makeatletter
\g@addto@macro{\endabstract}{\@setabstract}
\newcommand{\authorfootnotes}{\renewcommand\thefootnote{\@fnsymbol\c@footnote}}%
\makeatother

\usepackage{graphpap, color}
\usepackage{mathrsfs}
\usepackage{pstricks}
\usepackage{color}
\usepackage{cancel}
\usepackage[mathscr]{eucal}

\usepackage{pstricks}
\usepackage{color}
\usepackage{cancel}
\usepackage{verbatim}
\usepackage{latexsym}
\usepackage[all]{xy}

\usepackage{upgreek}

 \usepackage{wasysym}
\usepackage{ esint }

\def\index{{\mathbb I}}

\def\virt{^{\vir}}

\numberwithin{equation}{section}

\def\gm{\CC\sta}

\def\sO{{\mathscr O}} \def\sC{{\mathscr C}}  \def\sM{{\mathscr M}}
\def\sN{{\mathscr N}}
\def\sL{{\mathscr L}}

\def\sO{\mathscr{O}}

\def\sE{\mathscr{E}}

\def\sU{\mathscr{U}}

\newcommand{\CC}{\mathbb{C}}

\newcommand{\EE}{\mathbb{E}}

\newcommand{\PP}{\mathbb{P}}
\newcommand{\QQ}{\mathbb{Q}}

\newcommand{\ZZ}{\mathbb{Z}}

\newcommand{\bL}{\mathbf{L}}

\def\redd{{\mathrm{red}}}

\def\Gm{{G}}

\newcommand{\vir}{ {\mathrm{vir}} }
\newcommand{\ev}{ \mathrm{ev} }

\newcommand{\vGa}{\Ga}

\def\Wfix{\cW_\Ga} 
\def\fixW{\cW^\fix}

\newcommand{\cal}{\mathcal}

\def\cB{{\cal B}}

\def\cC{{\cal C}}
\def\cD{{\cal D}}

\def\cL{{\cal L}}
\def\cM{{\cal M}}
\def\cN{{\cal N}}
\def\cO{{\cal O}}
\def\cP{{\cal P}}

\def\cV{{\cal V}}
\def\cW{{\cal W}}

\def\cZ{{\cal Z}}

\def\dual{^{\vee}}

\def\sta{^\ast}
\def\st{^{\mathrm{st}}}
\def\virt{^{\mathrm{vir}}}

\def\sta{^{\ast}}

\def\sta{^*}

\def\lbe{_{\beta}}
\def\lra{\longrightarrow}

\def\lsta{_{\ast}}

 \def\expdim{\mathrm{exp.}\dim}

\newcommand{\Si}{\Sigma}
\newcommand{\Ga}{\Theta}

\def\begeq{\begin{equation}}
\def\endeq{\end{equation}}
\def\and{\quad{\rm and}\quad}
\def\bl{\bigl(}
\def\br{\bigr)}
\def\defeq{:=}

\def\sub{\subset}
\def\Ao{{\mathbb A}^{\!1}}

\def\Po{{\mathbb P^1}}
\def\and{\quad\text{and}\quad}

\def\lalp{_\alpha}

\def\reg{{\mathrm{reg}}}

 \DeclareMathOperator{\Aut}{Aut}

\DeclareMathOperator{\spec}{Spec}

\def\lggd{_{g,\gamma,\bd}}

\def\cWgg{\cW\lggd}
\newcommand{\pre}{ {\mathrm{pre}} }

\newtheorem{prop}{Proposition}[section]
\newtheorem{proposition}[prop]{Proposition}
\newtheorem{theo}[prop]{Theorem}
\newtheorem{lemm}[prop]{Lemma}

\newtheorem{rema}[prop]{Remark}
\newtheorem{exam}[prop]{Example}
\newtheorem{defi}[prop]{Definition}
\newtheorem{definition}[prop]{Definition}

\newtheorem{cons}[prop]{Construction}

\newtheorem{defi-prop}[prop]{Definition-Proposition}
\newtheorem{defi-theo}[prop]{Definition-Theorem}

\def\bone{{\mathbf 1}}

\def\Ob{\cO b}

\def\bul{^\bullet}

\def\ev{\text{ev}}

\def\sO{{\mathscr O}}

\def\sD{{\mathscr D}}

\def\beq{\begin{equation}}
\def\eeq{\end{equation}}

\def\vsp{\vskip5pt}
\def\Pf{{\PP^4}}

\def\bee{\begin{equation}}
\def\eeq{\end{equation}}

\def\sC{{\mathscr C}}

\def\bd{{\mathbf d}}

\let\eps=\epsilon

\def\ti{\tilde}

\def\barM{{\overline{M}}}

\def\fix{G} 

\def\AA{\mathbb A}

 \def\gnd{_{g,n,d}} \def\gnD{_{g,n,\bd}}

 \def\utw{^{{\mathrm tw}}}

\def\Gm{G}
\def\bmu{{\boldsymbol \mu}}
\def\cWggm{\cW\lggd^{-}}
\def\nggd{_{g,\gamma,\bd}} \def\nggdm{^{-}_{g,\gamma,\bd}}
\def\ga{\Theta} \def\ufl{^{\mathrm{fl}}}

 \def\pt{\mathrm{pt}}
\def\lred{_{\mathrm{red}}}

\def\bmu{{\boldsymbol \mu}}

\def\umv{^{\mathrm{mv}}}
\def\ufl{^{\mathrm{fl}}}
\def\ualp{^\alpha}  \def\ualpbe{^{\alpha\beta}}
\def\ufix{^G} 
\def\umv{^{\mathrm{mv}}}

 \def\fl{{\mathrm{fl}}}
\def\lrga{_{(\ga)}}
\def\Up{\Upsilon}

\def\fa{\mathfrak a}

\def\ureg{^{\mathrm{reg}}}
\def\lgnd{_{g,n,\bd}}

\def\Cont{\mathrm{Cont}}

\begin{document}

\title[N-Mixed-Spin-P fields]{The theory of N-Mixed-Spin-P fields}

\author[Huai-Liang Chang]{Huai-Liang Chang}
\address{Department of Mathematics, Hong Kong University of Science and Technology, Hong Kong} \email{mahlchang@ust.hk}
\thanks{${}^1$Partially supported by   Hong Kong GRF16301515 and  GRF16301717.}

\author[Shuai Guo]{Shuai Guo$^2$}
\address{School of Mathematical Sciences and Beijing International Center for Mathematical Research, Peking University, China
}
\email{guoshuai@math.pku.edu.cn}
\thanks{${}^2$Partially supported by NSFC grants 11431001 and 11890661.}

\author[Jun Li]{Jun Li$^3$}
\address{
Shanghai Center for Mathematical Sciences, Fudan University, China} \email{lijun2210@fudan.edu.cn}
\thanks{${}^3$Partially supported by  NSF grant DMS-1564500 and DMS-1601211.}

\author[Wei-Ping Li]{Wei-Ping Li$^4$}
\address{Department of Mathematics, Hong Kong University of Science and Technology, Hong Kong} \email{mawpli@ust.hk}
\thanks{${}^4$Partially supported by by Hong Kong   GRF16301515 and GRF 16303518.}

\maketitle
%

\begin{abstract}  This is the first part of the project toward proving the BCOV's Feynman graph sum formula of all genera Gromov-Witten invariants
of quintic Calabi-Yau threefolds. In this paper, we introduce the notion of N-Mixed-Spin-P fields, construct their moduli spaces, their virtual cycles, their virtual localization formulas, and a vanishing result associated with irregular graphs.
\end{abstract}


\setcounter{tocdepth}{1}
\tableofcontents

\section{Introduction}
This is the first of a three-paper series. In this paper, we introduce N-Mixed-Spin-P fields (NMSP fields),
construct their moduli spaces,
construct their potential functions, and prove that these potential functions admit a decomposition separating them
into a part that involves the Gromov-Witten (GW) invariants of the quintic Calabi-Yau (CY) threefolds,
and a part that
involves the Fan-Jarvis-Ruan-Witten (FJRW) invariants of the Fermat quintic polynomial.

In this three-paper series, we will prove several structural  results of the GW potential function $F_g$.
In \cite{NMSP2}, we will prove Yamaguchi-Yau's finite generation conjecture for $F_g$, which implies  the convergence of $F_g$. In \cite{NMSP3},  we will prove the Bershadsky-Cecotti-Ooguri-Vafa (BCOV) Feynman graph sum formula. One of its consequences is the Yamaguchi-Yau Functional equation \cite[Sect.~3.3]{YY} (for the proof see \cite[Sect.~7]{NMSP3}).
Our method is expected to apply to any complete intersection Calabi-Yau threefolds in products of weighted projective spaces.

\smallskip
The notion of MSP field was introduced in \cite{CLLL,CLLL2} as a mathematical theory to realize the phase
transition for GW invariants of quintic CY threefolds envisioned by Witten \cite{GLSM}.
The relations derived from the vanishings via their virtual cycles  
have produced, or reproduced, results on low genus GW invariants of the quintic CY threefolds and the
low genus FJRW invariants of the Fermat quintic polynomials (cf. \cite{CGLZ, GR}).

In the current project, we modified MSP fields: we duplicate one of the fields by $N$
copies, and call the resulting field an NMSP field.
 Duplicating one field by $N$ copies provides a $G=(\CC\sta)^N$-equivariant theory,
 which simplifies significantly the structure of the generating series in the package: For one, when $N$ is large,
the twisted class that appears in the localization
formula at the quintic fixed loci becomes the original quintic virtual class; secondly, it greatly simplifies
the chain and tail contributions in the localization formula (c.f. \cite{NMSP2}), and realizes BCOV Feynman propagators as NMSP two point functions (c.f. \cite{NMSP3}).
All of these play a significant role in allowing us to prove the BCOV Feynman rules for the GW potentials of quintic CY threefolds. 

\smallskip

We now briefly outline the structure of the paper. In this paper, we focus on the Fermat quintic $x_1^5+\cdots+x_5^5$.
An NMSP field of numerical date $(g,\gamma,\bd)$ of the Fermat quintic polynomial is 
\beq\label{xi0}
\xi=(\sC,\Si^\sC,\sL,\sN, \varphi,\rho,\mu,\nu);\ \   \varphi=(\varphi_i)_{i=1}^5,\ \mu=(\mu_j)_{j=1}^N,
\eeq
where $\Si^\sC\sub \sC$ is a genus $g$ $\ell$-pointed twisted curve;  $\sL$ and $\sN$ are invertible sheaves of
$\sO_\sC$-modules; the  monodromy of $\sL$  along the $i$-th marking $\Sigma_i\sub\Sigma^\sC$ is given by $\gamma_i$;
and $\varphi$, $\rho$, $\mu$, and $\nu$ are fields
$$\varphi_i\in H^0(\sL),\ \rho\in H^0(\sL^{-5}\otimes\omega_\sC(\Sigma^\sC)),\ \mu_j\in H^0\big(\sL\otimes\sN),\
\nu\in H^0(\sN),
$$
satisfying certain properties. It is called an NMSP field because $\mu$ consists of $N$ sections
in the same space.
When $N=1$ it is an MSP field (cf. Definition \ref{def-curve}).
\smallskip

We fix $N\ge 1$, and let $\cW\lggd$ be the moduli of stable NMSP fields of type $(g,\gamma,\bd)$,
where $g$ is the genus of the domain curves, $\gamma=(\gamma_i)_{i=1}^\ell$ is the monodromy data,
and $\bd=(d_0,d_\infty)$
is the degree of $\sL\otimes\sN$ and $\sN$. The moduli $\cW\lggd$ is a ${G}$-stack,
via the standard action of $G$ on the $N$-components of $\mu$.

\begin{theo}\label{thm1}
The moduli space $\cW\nggd$ is a separated DM stack, locally of finite type.
It has a cosection localized ${G}$-equivariant virtual cycle $[\cW\nggd]\virt$,
lying in a proper $G$-invariant substack $\cW\nggdm\sub \cW\nggd$:
$$[\cW\nggd]\virt\in A^{\Gm}\lsta (\cW\nggdm).
$$
\end{theo}

We define the NMSP potential function. 
 In this article, we will confine ourselves to the case where the markings consist of $n$ scheme
markings decorated by $(1,\rho)$, meaning that the  field $\rho$ vanishes along the
markings. Namely we shall always choose $\gamma=(1,\rho)^n$.  We sometimes denote by $\cW\lgnd$ the moduli of stable NMSP fields $\cW\lggd$
of numerical data $(g,(1,\rho)^n,\bd)$. 

Then for any $\xi\in \cW\lgnd$, as in \eqref{xi0}, and the $i$-th marking $\Sigma_i\sub \Sigma^\sC$,
the definition of NMSP fields forces $\nu|_{\Sigma_i} $ nowhere vanishing, and
$$0\ne \bl\varphi_1|_{\Sigma_i},\cdots,\varphi_5|_{\Sigma_i},(\mu_1/\nu)|_{\Sigma_i},\cdots,(\mu_N/\nu)|_{\Sigma_i}\br
\in H^0(\sL|_{\Sigma_i})^{\oplus (5+N)}.
$$
Thus it defines a $G$-equivariant evaluation morphism
\beq\label{evi}
\ev_i: \cW\lgnd\lra \PP^{4+N},
\eeq
where $G$ acts on $\PP^{4+N}$ by scaling the last $N$-homogeneous coordinates of $\PP^{4+N}$.
\smallskip

Let $t_1,\cdots,t_N$ be the standard generators of $A_G^2(\pt)$. For $1\le i\le n$, we let $\tau_i\in H_G\sta(\PP^{4+N})$,
 and let $\bar\psi_i\in H^2_G(\cW\nggd)$ be the ancestor psi-class (i.e. the pullback from $\barM_{g,n}$). We let
$\zeta_N=\exp(\frac{2\pi \sqrt{-1}}{N})$.
We define
\beq\label{nmsp0}
\bigl<\prod_{i=1}^n \tau_i\bar\psi_i^{k_i}\bigr>^M_{g,n,d_\infty}=\sum_{d\ge 0}(-1)^{d+1-g}q^d\cdot
\int_{[\cW_{g,n,(d,d_\infty)}]\virt}\prod_{i=1}^n \ev_i\sta(\tau_i)\bar\psi_i^{k_i}.
\eeq
\medskip
\noindent
{\bf Convention}. {\sl
After equivariant integration, we will substitute $t\lalp=
-\zeta_N^\alpha t$ for a formal variable $t$, where $\alpha$ is always understood to be an integer in $[1,N]$.
}


\begin{theo}[Polynomiality of NMSP potential]\label{thm2}
We suppose all $\tau_i$ have pure degree\footnote{Our convention is that elements in $H^2$ has degree 1.};
and  let $\eps= \frac{1}{N} \big(\sum^n_{i=1}(\deg \tau_i+k_i-1) -d_\infty \big)$. Then 
\beq\label{111}
t^{-(g-1+\eps)N }\cdot \bigl<\prod_{i=1}^n \tau_i\bar\psi_i^{k_i}\bigr>^M_{g,n,d_\infty}
\eeq
is a rational polynomial in $q/t^N$ of degree no more than $g-1+\eps$.
\end{theo}

Notice that \eqref{111} vanishes unless $\eps\in \mathbb Z$, as we have substituted $t\lalp$ by
$-\zeta_N^\alpha t$. In \cite{NMSP2}, we shall apply the theorem for sufficiently large $N$ so that  the polynomial \eqref{111} has degree $\leq g-1$.  

\medskip
We will use virtual localization (cf. \cite{GP}) to study the structure of this correlator function.
For a narrow $(g,\gamma,\bd)$, we will use flat decorated graphs $\ga\in G\lggd^\fl$ to parameterize
open and closed substacks of the $G$-fixed part
$$ (\cW\lggd)^\fix=\coprod_{\ga\in G\lggd^\fl} \cW_{(\ga)}.
$$
The virtual localization gives 
\beq\label{vir-loc}
[\cW\lggd]\virt= \sum_{\ga \in G\lggd^\fl}  \frac{[\cW_{(\ga)}]\virt}
{e(N\virt_{\ga})} ,
\eeq
where  $[\cW_{(\ga)}]\virt\in A\lsta(\cW_{(\ga)}\cap\cW\lggd^-)$.

We briefly describe these graphs. A graph $\ga\in G\lggd^\fl$ has its vertices possibly lying in three different levels: level $0$, level $1$, and level $\infty$.
In the virtual localization formula, a level 0 vertex will contribute a GW invariant of the
quintic threefold; a level 1 vertex will contribute a Hodge integral on the moduli of pointed curves, and
a level $\infty$ vertex will contribute an FJRW invariant of the Fermat quintic polynomial.

 We find that, among all graphs, a subclass of ``irregular graphs"  contributes zero numerically.
In short, a graph $\Theta$ is irregular if it has an edge incident to two vertices,  one at level $0$ and one at
level $\infty$, or has some special vertices at level $\infty$ (c.f. Definition \ref{de-regu}). We prove the irregular vanishing theorem.

\begin{theo}\label{thm3}
For an irregular $\Theta\in G\lggd\ufl$ not a pure loop, $[\cW_{(\Theta)}]\virt\sim 0$.
\end{theo}

Here a graph is called a pure loop if it has no legs, no stable vertices, and every vertex has
exactly two edges attached to it. And $\sim$ means that for the purpose of numerical applications, one can almost always treat
any two classes $A\sim A'$ as identical classes. (See Definition \eqref{sim}.)
We let $G\lggd^\reg$ be the set of
regular graphs in $G\gnd^\fl$.

\medskip
For the case when $\gamma$ consists of $n$ $(1,\rho)$-type markings, we will use $\cW\lggd=\cW\gnD$.
For classes $\tau_i\in H\sta_G(\PP^{4+N})$, coupled with \eqref{vir-loc} and Theorem \ref{thm3}, we obtain
\beq\label{localization0}
\pi\lsta \bl [\cW\gnD]\virt\cdot(-)\br=\sum \pi\lsta\Bigl(
\frac{[\cW_{(\Theta)}]\virt}{e_G(N\virt_\Theta)}\cdot(-)\Bigr)
:=\sum \pi\lsta\bl \Cont_\Theta\cdot (-)\br,
\eeq
where the summation is over all $\Theta\in G\gnD^\reg$, the map $\pi:\cW\gnD^-\to\pt$ is the projection, and
$(-)=\ev_1\sta\tau_1\psi_1^{k_1}\cdots\ev_n\sta\tau_n\psi_n^{k_n}$. Namely when we calculate correlators using \eqref{vir-loc}, we only need to sum over regular graphs, due to Theorem \ref{thm3}.
\smallskip

In the end, we will use the knowledge of the localization formula to prove a decomposition formula, relating
$\Cont_\Theta$ with the localization contribution from the graph $\Theta'$ obtained by dislodging an edge from one of its
incidental vertex.


\medskip
The organization of this paper is as follows. In \S2 and \S 3, we will state the basic properties of NMSP
fields, and prove the basic properties of the moduli $\cW\nggd$, including Theorem \ref{thm1}. In \S 4, we
will investigate the structure of the fixed loci $(\cW\nggd)^\fix$. We will prove Theorem \ref{thm3} in \S 5.
In \S 6,  we will provide the details of the virtual localization formula.
We will prove Theorem \ref{thm2} in the end.

\medskip
\noindent
{\bf Acknowledgement}.
The authors would like to take this opportunity to thank Yongbin Ruan for generously sharing his understanding of
the insight from physics on GW invariants of quintic CY threefolds.
The authors also thank Melissa Liu for her help to this project.

\section{NMSP fields}

In this section, we introduce the notion of NMSP fields and prove their basic properties.
The proofs of most of the stated properties can be found in \cite{CLLL}.
We only consider NMSP fields of Fermat type $x_1^5+\cdots+x_5^5$ in this paper.

\smallskip
We let $\bmu_5\le \CC\sta$ be the subgroup of the $5$-th roots of unity. We let
$$\bmu_5^{\text{br}}=\bmu_5\cup \{(1,\rho),(1,\varphi)\},\and 
\bmu_5^{\text{na}}=\bmu_5^{\text{br}}-\{1\}. 
$$
For $\eta\in\bmu_5$, we let $\langle \eta\rangle\le\gm$ be the subgroup generated by $\eta$. For the two exceptional elements $(1,\rho)$ and $(1,\varphi)$, we agree
$\langle (1,\rho)\rangle=\langle (1,\varphi)\rangle=\langle 1\rangle$. 

We call a triple $(g,\gamma,\bd)$, where
$$g\ge 0,\quad \gamma=(\gamma_1,\cdots,\gamma_\ell)\in (\bmu_5^{\text{br}})^{\times\ell},\and \bd=(d_0, d_\infty)\in
\QQ^{\times 2},
$$
a numerical data of NMSP fields.
In case $ \gamma\in (\bmu_5^{\text{na}})^{\times\ell}$,
we say that $(g,\gamma,\bd)$ or $\gamma$ is
{\sl narrow}. Otherwise, when some $\gamma_i=1$
we say $\gamma_i$, $\gamma,$ and $(g,\gamma,\bd)$ are  {\sl  broad}.

Recall the convention on invertible sheaves on a twisted curve.
An indexed  $\bmu_5$-marking of a twisted curve has a local model given by
$$[\text{pt}/\bmu_5]\sub \sU\defeq [\AA^1/\bmu_5]=\left[ \spec\CC[u]\big/ \bmu_5\right]. 
$$
The invertible sheaf of $\sO_{\sU}$-modules $\sO_{\sU}(\frac{m}{5}p)$ has monodromy $\zeta_5^m$,
or index $m\in [0,4]$, where as always $\zeta_5=e^{\frac{2\pi\sqrt{-1}}{5}}$.

For an $S$-family of pointed twisted nodal curves $\Si^\sC\sub \sC$, we let
$\omega^{\log}_{\sC/S}=\omega_{\sC/S}(\Si^\sC)$; for
$\alpha\in \bmu_5^{\text{br}}$, we let $\Si^\sC_\alpha=\coprod_{\gamma_i=\alpha}\Si^\sC_i$.

\begin{definition}\label{def-curve} Let $S$ be a scheme and  $(g,\gamma,\bd)$ be a numerical data.
An $S$-family of NMSP-fields of type $(g,\gamma,\bd)$ is
\beq\label{field}
\xi=(\sC,\Si^\sC,\sL,\sN, \varphi,\rho, \mu,\nu)
\eeq
such that
\begin{enumerate}
\item[(1)] $\cup_{i=1}^\ell\Si_i^\sC = \Si^\sC\subset \sC$ is an $\ell$-pointed, genus $g$, fiberwise connected
$S$-twisted curve so that the $i$-th marking $\Si^\sC_i$ is banded by $\langle\gamma_i\rangle\le \gm$;
\item[(2)] $\sL$ and $\sN$ are representable invertible sheaves on $\sC$; $\sL\otimes \sN$ and $\sN$ have fiberwise degrees $d_0$ and $d_\infty$ respectively; the monodromy of $\sL$ along $\Si^\sC_i$ is
$\gamma_i$ when $\langle \gamma_i\rangle\ne\langle 1\rangle$;
\item[(3)] $\mu=(\mu_1,\cdots,\mu_N)$, $\mu_i\in H^0( \sL\otimes\sN)$; $\nu \in H^0( \sN)$, and $(\mu,\nu)$ is nowhere zero;
\item[(4)]
$\varphi\in H^0(\sL(-\Si^\sC_{(1,\varphi)}))^{\oplus 5}$; and $(\varphi, \mu)$ is nowhere zero;
\item[(5)] $\rho \in H^0(\sL^{- 5}\otimes \omega^{\log}_{\sC/S}(-\Si^\sC_{(1,\rho)}))$;
and $(\rho, \nu)$ is nowhere zero. 
\end{enumerate}
We call $\xi$ narrow (resp. broad) if $\gamma$ is narrow (resp. broad).
\end{definition}

We define an arrow $\xi'=(\sC' ,\Si^{\sC'} , \sL',\cdots)\to \xi$
between two $S$-NMSP-fields to be a 3-tuple $(a,b,c)$ such that
$a:(\Si^{\sC'} \subset \sC')\to  (\Si^\sC \subset \sC)$ is an $S$-isomorphism of twisted curves, $b$ and $c$ are
isomorphisms of $\sL'$ and $\sN'$ with $\sL$ and $\sN$, respectively, which  preserve all fields
(cf. \cite[Definition 2.6]{CLLL}).


We define $\cW\lggd^{\pre}$ to be the category fibered in groupoids over
the category of schemes, such that objects in $\cW\lggd^{\pre}(S)$ are
$S$-families of NMSP-fields of type $(g,\gamma,\bd)$, and morphisms are given as above.

\begin{defi} We call $\xi\in \cW\lggd^{\pre}(\CC)$ {\em stable} if $\Aut(\xi)$ is finite.
We say $\xi\in \cW\lggd^{\pre}(S)$ is stable if $\xi|_s$ is stable for every closed point $s\in S$.
\end{defi}

It is direct to check that being stable is an open condition.
Let $\cW\lggd\subset \cW\lggd^{\pre}$ be the open substack of families of stable objects in $\cW\lggd^{\pre}$.
The group $G=(\gm)^{N}$ acts on $\cW\lggd$ tautologically:
\beq\label{Gm}
t\cdot (\Si^\sC, \sC, \sL, \sN,\varphi,\rho, \mu, \nu)
=  (\Si^\sC, \sC, \sL, \sN,\varphi,\rho,t\cdot\mu,\nu),\quad t\in \Gm,
\eeq
where 
$t\cdot (\mu_1,\cdots,\mu_N)=(t_1\mu_1,\cdots,t_N\mu_N)$.

\begin{theo}
The stack $\cW\lggd$
is a DM ${G}$-stack, locally of finite type.
\end{theo}

\begin{proof}
The theorem follows immediately from that the stack $\cM^{\text{tw}}_{g,\ell}$ of
stable twisted $\ell$-pointed curves is a DM stack, and each of its connected components
is proper and of finite type (see \cite{AJ, O}). 
\end{proof}

In this paper, we will reserve the symbol ${G}$ for this action on $\cW\lggd$.

\begin{exam} [Stable maps with $p$-fields]\label{pfield}
An NMSP-field $\xi\in\cW\lggd(\CC)$ having $\mu=0$ will have  $\sN\cong \sO_\sC $ and $\nu=1$.
Then $\xi$ is reduced to  a stable map $f=[\varphi]: \Si^\sC\sub \sC\to\PP^{4}$
together with a $p$-field $\rho\in H^0(f\sta \sO_{\PP^{4}}(-5)\otimes\omega_\sC^{\log})$.
Moduli of genus $g$ $\ell$-pointed stable maps to $\PP^{4}$ with
$p$-fields, denoted by $\barM_{g,\ell}(\PP^{4},d)^p$, was first introduced in \cite{CL}.
\end{exam}

\begin{exam} [Stable 5-Spin maps with $p$-fields]\label{spinp}
An NMSP-field $\xi\in\cW\lggd(\CC)$ having $\nu=0$ will have $\rho:\sL^5\cong\omega^{\log}_\sC$, and
that $\mu$ defines a degree $d_0$ morphism $[\mu]:\sC\to \PP^{N-1}$ ($d_0=\deg(\sL\otimes\sN)$).
This way, we obtain a morphism 
\beq\label{nu2}
\cW\lggd\cap\{\xi\mid\nu=0\}\lra \cM_{g,\gamma}^{1/5,5p}\times_{\cM_{g,\gamma}\utw}
\barM_{g,\gamma}(\PP^{N-1},d_0).
\eeq
Here, $\cM_{g,\gamma}^{1/5,5p}$ is the stack of $5$-spin curves (not necessarily
stable) with fixed monodromy $\gamma$ and
$5$ $p$-fields, $\cM_{g,\gamma}\utw$ is the moduli of genus $g$ twisted nodal curves of marking $\gamma$, and
$\barM_{g,\gamma}(\PP^{N-1},d_0)$ is the moduli of stable morphisms.
The ${G}$-action on \eqref{nu2} is via its obvious action on $\barM_{g,\ell}(\PP^{N-1},d_0)$.
When $p$-fields are zero, this moduli space was studied in \cite{JKV}.
\end{exam}


\subsection{Cosection localized virtual cycle}\label{sub2.3}
In the following, we will only consider narrow NMSP fields unless otherwise mentioned.

The DM stack $\cWgg$ admits a tautological $\Gm$-equivariant perfect obstruction theory.
Let $\cD$ be the stack of data $(\sC,\Si^{\sC},\sL,\sN)$, where $\Si^{\sC}\sub\sC$ are
pointed connected twisted curves, 
and $\sL$ and $\sN$ are invertible sheaves on $\sC$, with $\sL\oplus\sN$ representable.
Clearly, $\cD$ is a smooth Artin stack, locally of finite type. 

Let
\beq\label{universal}
\pi: \Si^\cC\sub \cC\to\cWgg\quad \text{with}\quad  (\cL,\cN,\varphi,\rho, \mu,\nu)
\eeq
be the universal family of $\cWgg$. Define
$q: \cWgg\to\cD$
to be the forgetful morphism, forgetting the fields $(\varphi,\rho, \mu,\nu)$. 
The morphism $q$ is $\Gm$-equivariant with $\Gm$ acting on $\cD$ trivially.


We adopt the following convention throughout this paper. We abbreviate
$\cP=\cL^{-5}\otimes\omega^{\log}_{\cC/\cWgg}$. For $\gamma_i\in \bmu_5$, we set $m_i\in [0, 4]$  so that
$\gamma_i=\zeta_5^{m_i}$. Let $\ell_\varphi$ 
be the number of $\gamma_i$ so that $\gamma_i=(1,\varphi)$; likewise for $\ell_\rho$. We let
$\ell_0=\ell-\ell_\varphi-\ell_\rho$. 



\smallskip

We first construct its perfect obstruction theories \cite{BF, LT}.

\begin{proposition}\label{dim}
The pair $q: \cWgg\to\cD$ comes with a $\Gm$-equivariant relative perfect obstruction theory
$$E_{\cW/\cD}\lra L\bul_{\cWgg/\cD},
$$
where
$$E_{\cW/\cD}=\Bigl( R\pi\lsta\bl \cL(-\Si^\cC_{(1,\varphi)})^{\oplus 5}\oplus \cP(-\Si^\cC_{(1,\rho)})\oplus
(\bigoplus_\alpha\cL\otimes\cN\otimes\bL_\alpha)\oplus \cN\br\Bigr)\dual,
$$
where $\oplus_\alpha$ is a summation from $\alpha=1$ to $N$, and
$\bL_\alpha$ is the 1-dimensional $G$-representation where the $\alpha$-th factor of $G$ acts with weight one,
while other factors act trivially. Then the virtual dimension of $\cWgg $ is
\beq\label{vdim}
\delta(g,\gamma,\bd)\defeq
Nd_0+d_\infty+N(1-g)+\ell -4\big(\ell_\varphi+\sum_{i=1}^{\ell_0}\frac{m_i}{5}\big).
\eeq
\end{proposition}

\begin{proof}
The proof is exactly the same as in \cite{CLLL}.
\end{proof}

The associated relative obstruction sheaf of $q:\cWgg\to \cD$ is
$$\Ob_{\cWgg/\cD}= R^1\pi\lsta\bl\cL(-\Si^\cC_{(1,\varphi)})^{\oplus 5}\oplus
\cP( -\Si^\cC_{(1,\rho)})\oplus
(\oplus_\alpha\cL\otimes\cN\otimes\bL_\alpha)\oplus \cN\br;
$$
the absolute obstruction sheaf $\Ob_{\cWgg}$
is defined by the exact sequence
\beq\label{quot}
q\sta T_{\cD}\lra \Ob_{\cWgg/\cD}
\lra \Ob_{\cWgg}\lra 0.
\eeq
We define a cosection (homomorphism), for narrow $\gamma$, 
\beq\label{co-1}
\sigma: \Ob_{\cWgg/\cD}\lra \sO_{\cWgg}
\eeq
by the rule that at an $S$-point $\xi\in\cWgg(S)$ 
as in \eqref{field}, and for
$$
(\dot\varphi,\dot\rho,\dot\mu,\dot\nu)\in
H^1\bl \sL(-\Si^\sC_{(1,\varphi)})^{\oplus 5}\oplus
\sL^{-5}\otimes\omega^{\log}_{\sC/S} (-\Si^\sC_{(1,\rho)})\oplus(\oplus_\alpha\sL\otimes\sN\otimes\bL_\alpha) 
\oplus \sN\br,
$$
\beq\label{mixed-cosection}
\sigma(\xi) 
(\dot\varphi,\dot\rho,\dot\mu,\dot\nu) = 5\rho\sum \varphi_{i}^{4}\dot\varphi_{i}
+\dot\rho\sum\varphi_i^5 \in H^1(\omega_{\sC/S})\equiv H^0(\sO_\sC)\dual.
\eeq
Here the term $5\rho\sum \varphi_{i}^{4}\dot\varphi_{i}
+\dot\rho\sum\varphi_i^5$ lies in $H^1(\sC,\omega_{\sC/S})$ because when $\gamma_i\neq (1, \rho)$,
$\varphi_j|_{\Si^{\sC}_i}=0$.


\begin{lemm}\label{lem:co} Let $\gamma$ be narrow. Then
the rule \eqref{mixed-cosection} defines a $\Gm$-equivariant cosection $\sigma$ as in \eqref{co-1}.
It lifts to a $\Gm$-equivariant homomorphism $\Ob_{\cWgg}\to \sO_{\cWgg}$.
\end{lemm}

We call the lifted homomorphism a cosection  of $\Ob_{\cWgg}$. The proof of the lemma is exactly the same as that of
\cite[Lemma 2.12]{CLLL}, and will be omitted.

\smallskip

As in \cite{KL}, we define the degeneracy locus of $\sigma$ to be the (reduced) substack $\cWggm\sub\cWgg$
whose closed points are 
\begin{eqnarray}\label{deg-loci}
\cWggm(\CC)=\{\xi\in \cWgg(\CC)\mid \sigma|_\xi=0\}.
\end{eqnarray}
We have the following criterion,
whose proof is the same as in \cite[Lemma 2.13]{CLLL}.

\begin{lemm}\label{degenerate-locus} Let $\gamma$ be narrow. A $\xi\in \cWgg(\CC)$
lies in $\cWggm(\CC)$ if and only if
\beq
(\varphi=0 )\cup ( \varphi_1^5+\cdots+\varphi_5^5=\rho=0)=\sC.
\eeq
\end{lemm}

Applying  \cite{KL,CoVir}, we obtain the cosection localized virtual cycle
$$
[\cWgg]\virt \in A^G_{\delta} \bl \cWggm\br,\quad \delta=\delta(g,\gamma,\bd).
$$

\section{Properties of the moduli of NMSP fields}
\def\Slsta{{S\lsta}}

In this section, we fix a narrow $(g,\gamma, \bd)$; we abbreviate $\cW=\cWgg$, and consequently
abbreviate $\cW^-=\cWggm$. We prove that $\cW^-$ is proper.

\subsection{Stability criterion}\label{Sub3.1}
We denote by $\eta_0\in S$ a closed point in an affine smooth curve,
and denote $S\lsta=S-\eta_0$ its complement.

In using valuative criterion to prove properness, we need to take a finite base change $S'\to S$
ramified over $\eta_0$. By shrinking $ S$ if necessary,
we assume that there is an \'etale $S\to \Ao$ so that $\eta_0$ is the only point lying over $0\in\Ao$.
In this  way, for any positive integer $a$,  we can form
$S_a=S\times_{\Ao} \Ao$, where $\Ao\to \Ao$ is via $t\mapsto t^a$. 
Thus  $\eta'_0\in S_a$ lying over $\eta_0\in S$ is the only closed point lying over $0\in \Ao$, of ramification index $a$.

We will use subscript ``$\ast$" to denote families over $S\lsta$.  Hence $\xi\lsta\in\cW^{\text{pre}}(S\lsta)$($=
\cW^{\text{pre}}_{g,\gamma,\bd}(S\lsta)$)
takes the form $\xi_*=\big(\Si^{\sC_\ast}, \sC_*, \sL_*,\sN\lsta, \cdots\big)$.

\vsp
We first characterize stable objects in $\cW^{\text{pre}}(\CC)$. We say that an irreducible component $\sE\sub\sC$
is a rational curve if it is smooth and its coarse moduli is isomorphic to $\Po$.
We say $(g,\gamma,\bd)$ is in the stable range if the following holds: when $g=0$ then $\ell\ge 3$ or $\bd\ne (0,0)$, and when $g=1$ then
$\ell\ge 1$ or $\bd\ne (0,0)$.

%

\begin{lemm}\label{stable-cri}
Let $(g,\gamma,\bd)$ be in the stable range and let $\xi\in \cW^{\mathrm{pre}-}(\CC)$.
Then $\xi$ is unstable if $\sC$ contains a rational curve $\sE$ such that one of the following holds:
\begin{enumerate}
\item  $\sE\cap (\Si^\sC\cup \sC_{\mathrm{sing}})$ contains two points, $\sL^{\otimes 5}|_\sE\cong \sO_\sE$, and $\sL\otimes\sN|_\sE\cong\sO_\sE$;
\item 
$\sE\cap (\Si^\sC\cup \sC_{\mathrm{sing}})$ contains one point,
and either $\rho|_\sE$ is nowhere zero and $\sL\otimes\sN|_\sE\cong\sO_\sE$, or $\sL|_\sE\cong \sN|_\sE\cong \sO_\sE$.
\end{enumerate}
\end{lemm}

\begin{proof} 
Let $\xi\in\cW^{\mathrm{pre}-}(\CC)$ be unstable. For each irreducible $\sE\sub\sC$,
let $\Aut_{\sE}(\xi)$ be the subgroup of $\Aut(\xi)$ leaving $\sE$ invariant. As argued in the proof of
\cite[Lemma 3.3]{CLLL}, there is an irreducible $\sE\sub\sC$ so that
the group $\text{Im}( \Aut_{\sE}(\xi)\to \Aut(\sE))$ is infinite.

In case $\sE=\sC$, then it is easy to see that $\xi$ is not stable only when $(g,\gamma,\bd)$ is not in stable range. Thus by assumption,
$\sE\ne\sC$, and then $\sE$ has arithmetic genus zero (thus smooth), and $\sE\cap \sC_{\text{sing}}\ne \emptyset$.

We now consider the case when $\sE\cap(\Si^\sC\cup \sC_{\text{sing}})$ contains one point,
say $p\in\sE$. Suppose $\rho|_\sE=0$, then the same argument as in the proof of \cite[Lemma 3.3]{CLLL} shows that
$\nu_2|_\sE$ is nowhere vanishing; $\sN|_\sE\cong \sO_\sE$, and $(\varphi, \mu)|_\sE$ is a nowhere vanishing section of $H^0(\sL^{\oplus (5+N)}|_\sE)$.
Since ${G}$ is infinity and $(\varphi, \mu)|_\sE$ is ${G}$-equivariant, this is possible only if
$\sL|_\sE\cong\sO_\sE$ and $(\varphi, \mu)|_\sE$ is a constant section. This is (2).

Suppose $\rho|_\sE\ne0$. We argue that $\sL\otimes\sN|_\sE\cong\sO_\sE$ and $\mu|_{\sE}$ is a constant section.
Indeed, since $\xi\in \cW^{\text{pre}-}\lggd(\CC)$, we have $\varphi|_\sE=0$, thus $\mu|_{\sE}$ is nowhere vanishing and
defines a morphism $[\mu]: \sC\to \PP^{N-1}$. Since ${G}$ is infinite, $[\mu]:\sE\to \PP^{N-1}$
is not stable, thus $[\mu]|_\sE$ is a constant map. This is  possible if and only if
$\sL\otimes\sN|_\sE\cong\sO_\sE$.

Once $\sL\otimes\sN|_\sE\cong\sO_\sE$ is established, the same argument as in the proof of \cite[Lemma 3.3]{CLLL} shows
that this must be case  (2).

For the case (1), by the similar argument in the proof of \cite[Lemma 3.3]{CLLL}, we obtain $\sL^5|_\sE\cong \sO_\sE$. If $\rho|_\sE\neq 0$, by the similar argument as above, we get
$\sL\otimes\sN|_\sE\cong\sO_\sE$. If $\rho|_\sE=0$, $\nu|_\sE$ must be a constant. Thus $\sN|_\sE\cong \sO_\sE$ and $\mu\in H^0(\sL^{\oplus N}|_\sE)$. Therefore we get a map $\sE\to \mathbb P^{4+N}$ induced by the sections $(\varphi_1,\ldots,  \varphi_5, \mu_1, \ldots, \mu_N)$. Since ${G}$ is infinite, the map must be a constant resulting $\sL|_\sE\cong \sO_\sE$.
In any case, we  have $\sL\otimes\sN|_\sE\cong\sO_\sE$. This is case (1).
\end{proof}

%

\subsection{Valuative criterion for properness}
We begin with a simple extension result.

\begin{lemm}\label{prop-extension0}
Let $\xi_*\in \cW^{\mathrm{pre}-}(\Slsta)$
be such that $\rho\lsta=0$. Then possibly after a finite base change, $\xi\lsta$ extends to a $\xi\in \cW^{\mathrm{pre}-}(S)$.
\end{lemm}
\begin{proof}
Since $\rho\lsta=0$, $\nu$ is nowhere vanishing and $\sN\lsta\cong\sO\lsta$.
Thus $(\varphi\lsta,  \mu\lsta)$ is a nowhere vanishing section of $H^0(\sL\lsta^{\oplus 5}\oplus \sL\lsta^{\oplus N})$,  and  induces a morphism
$(\Si^{\sC\lsta}, \sC\lsta)\to\mathbb P^{4+N}$ such that $(f\lsta)\sta\sO_{\mathbb P^{4+N}}(1)\cong\sL\lsta$.
By the stability assumption of $\xi\lsta$,
this morphism is an $S\lsta$-family of stable maps. Thus the Lemma follows from that the moduli of stable maps
to $\PP^{4+N}$ of fixed degree is proper.
\end{proof}

\begin{lemm}\label{prop-extension3}
Let $\xi_*\in \cW^{\mathrm{pre}-}(\Slsta)$
be such that $\varphi\lsta=\nu\lsta=0$. Then the conclusion of Lemma \ref{prop-extension0} holds.
\end{lemm}

\begin{proof}
Let $\xi\lsta=(\sC\lsta,\Si^{\sC\lsta},\sL\lsta,\cdots)$. Since $\nu_*=0$, $\rho_*$ is nowhere vanishing. Thus $(\sC\lsta, \Si^{\sC\lsta},\sL\lsta )$ is an $S\lsta$-family of spin curves. The section $\mu\lsta$ induces  a morphism
$$f\lsta \colon \sC\lsta\to \mathbb P^{N-1}$$
 such that  $f\lsta^*\sO(1)\cong \sL\lsta\otimes\sN\lsta$. This is exactly an $S_*$-family of stable spin maps studied by Jarvis, Kimura and Vantrob in \cite{JKV}. By the diagram on page 519 and Proposition 2.2.3 in \cite{JKV}, the moduli stack of stable spin maps is proper. Thus we can extend the spin curve $(\sC\lsta, \Si^{\sC\lsta},\sL\lsta )$ over $S\lsta$  to a spin curve $(\sC, \Si^{\sC},\sL )$ over $S$ with a stable map $f\colon \sC\to \mathbb P^{N-1}$ extending the map $f\lsta$. We define $\sN=f^*\sO(1)\otimes \sL^\vee$.

The family $(\sC,\Si^{\sC},\sL,\sN,\varphi=0, \rho,\mu,\nu=0)$ is a desired extension of $\xi\lsta$.
\end{proof}


\begin{lemm}\label{prop-extension1}Let $\xi\lsta\in \cW^{\mathrm{pre}-}(S\lsta)$ 
be such that $\sC_\ast$ is smooth, $\varphi\lsta=0$, $\rho\lsta\ne 0$ and $\nu\lsta\ne0$. 
Then possibly after a finite base change, we can extend $\xi\lsta$ to $\xi\in \cW^{\pre-}(S)$.
\end{lemm}

\begin{proof} Since $\varphi\lsta=0$, the section $\mu\lsta$ is nowhere vanishing and defines a (not necessariy stable) morphism
$[\mu\lsta]:\sC\lsta\to \PP^{N-1}$. By adding some auxiliary sections, possibly after a finite base change, we can extend
$\Si^{\sC\lsta}\sub\sC\lsta$ to $\Si^\sC\sub\sC$ over $S$, and extend $[\mu\lsta]$ to a morphism $[\mu]:\sC\to\PP^{n-1}$. In this way,  $\sM\defeq [\mu]\sta\sO(1)$ is an extension of $\sL\lsta\otimes\sN\lsta$  over $\sC$ so that $\mu_{\alpha\ast}\in H^0(\sL\lsta\otimes\sN\lsta)$ extends to $\mu_\alpha\in H^0(\sM)$, and
$[\mu]$ is the map defined by $(\mu_1,\cdots,\mu_N)$. 

Our next task is to extend $(\rho\lsta,\nu\lsta)$. For this we follow the proof in \cite[\S 3]{CLLL}.
Possibly after a further base change and desingularization, we can find an $S$-families of pointed twisted curves
$\Sigma^{\ti\sC}\sub \ti\sC$, a birational morphism $\phi:\ti\sC\to\sC$, such that $\phi$ induces an isomorphism of
$\Sigma^{\ti\sC}$ with $\Sigma^{\sC}$, and extend $\sL\lsta$ to an invertible sheaf $\ti\sL$ on
$\ti\sC$ so that, after letting $\ti\sM=\phi\sta\sM$, $\nu_{\ast}$ extends to a regular section
$\ti\nu\in H^0(\ti\sM\otimes\ti\sL\dual)$, and $\rho\lsta$ extends to a meromorphic section of
$\ti\sL^{-5}\otimes\omega_{\ti\sC/S}^{\log}$, so that a parallel statement of \cite[Proposition 3.5]{CLLL} holds.

After that, we follow the proof in \S  3 of \cite{CLLL}. We form a basket $\cB$ as in \cite[(3.3)]{CLLL},
using $(\rho=0)$ and $(\nu=0)$. Then the proof in \cite[(3.3)]{CLLL} can be line by line copied to this case to show that
we can choose a $\xi'\in \cW^{\pre-}(S)$ extending $\xi\lsta$.
This proves the lemma.
\end{proof}

\begin{lemm}\label{prop-extension2}Let $\xi\lsta\in \cW^{\mathrm{pre}-}(S\lsta)$ be as in Lemma \ref{prop-extension1}.
Then possibly after a base change, we can extend $\xi\lsta$ to a (stable) $\xi\in \cW^{-}(S)$.
\end{lemm}

\begin{proof}
Let $\xi$ be an extension given by Lemma \ref{prop-extension1}. In case $\xi_0$ is not stable,
 by Lemma \ref{stable-cri}, we have those rational curves $\sE\sub\sC_0$ satisfying (1) or (2) of the criterion.
By the same Lemma, we know $\sL\otimes\sN|_{\sE}\cong\sO_\sE$. Then we can apply the arguments in
\cite[\S 3.5]{CLLL} to show that, by blowing down and modifying fields, we can make $\xi$ stable.
\end{proof}

The case where $\sC\lsta$ is not necessarily irreducible can be treated by gluing.

\begin{prop}\label{proper-proof}Let $S\lsta$ be as before, and let $\xi\lsta\in \cW^{\mathrm{pre}-}(S\lsta)$.
Then possibly after a base change, we can extend $\xi\lsta$ to $\xi\in \cW^{-}(S)$.
\end{prop}

\begin{proof}
The proof is the same as that in \cite[{ \S 3}]{CLLL}.
\end{proof}

\subsection{Proof of the properness}\label{separatedness}

We continue to abbreviate $\cW=\cW\lggd$. We first prove that $\cW$ is separated.
As before, $\eta_0\in S$ is a closed point in a smooth curve over $\CC$,  and $S\lsta=S-\eta_0$.


\begin{lemm}\label{valuative2}
Let $\xi$, $\xi'\in \cW(S)$ be such that $\xi\lsta\cong \xi'\lsta\in \cW(S\lsta)$.
Then $\xi\cong\xi'$.
\end{lemm}

\begin{proof}Suppose $\sC\lsta$ is smooth.
The proof of the valuative criterion for separatedness in \cite[\S  4.1]{CLLL} line by line can be adopted to this
case, after we replace 
$\nu_1$ (resp. $\nu_2$) by $(\mu_1,\cdots,\mu_N)$ (resp. $\nu$).

Note that in \cite{CLLL}, whenever we know that $\sD\sub\sC$ is an irreducible component satisfying that $\nu_1|_\sD$
is nowhere vanishing, we conclude that $\sL\otimes\sN|_{\sD}\cong\sO_\sD$. This is no longer true, as assuming
$(\mu_1,\cdots,\mu_N)|_{\sD}$
is nowhere vanishing only provides us a morphism $[\mu]|_\sD:\sD\to\PP^{N-1}$, not necessarily constant.
In \cite[\S 4.1]{CLLL},   the fact that {\sl $\nu_1|_\sD$  nowhere
vanishing implies $\sL\otimes\sN|_{\sD}\cong\sO_\sD$} is used only at one place.  It is in the fifth paragraph of Sublamme 4.3.

To adopt to our case, we can modify the argument as follows. Firstly, we adopt the same notation for $\sD'$ and $D_k$ used in
Lemma 4.1 of \cite{CLLL}.
Using $\bar f\sta \rho'|_{D_k}=0$, 
we have $\rho'|_{\cD'}=\varphi'|_{\sD'}=0$, thus $(\mu_j'|_{\sD'})_{j=1}^N$ and $v'|_{\sD'}$ are
nowhere vanishing. Since $a_k+b_k=0$,
$\bar f\sta \mu'|_{\sD'-q'}=\bar\pi\sta (\mu|_x)|_{\sD'-q'}$. (Compare (4.4) in \cite{CLLL}.)
Thus $(\mu_j'|_{\sD'})_{j=1}^N$ defines a constant morphism $\sD'\to\PP^{N-1}$.
Because $\sD'$ contains one node and at most one marking of $\sC'_0$,
$\xi_0'$ becomes unstable, a contradiction.
This proves that $a_k=k$.

With this slight modification, the argument in \cite[\S  4.1]{CLLL} can be adopted here to prove this lemma,
assuming that $\sC\lsta$ is smooth.

Adopting the proof of \cite[Proposition 4.4]{CLLL}, one proves that 
the lemma holds without assuming that $\sC\lsta$ is smooth. This proves the lemma.
\end{proof}

\begin{prop} The stack $\cW$ is separated.
\end{prop}

\begin{proof} The proof is identical to that of \cite[{\S  4}]{CLLL}.
Because $\cW$ is locally of finite type, to prove that it is separated we only need to show that any finite type open substack
$\cW_0\sub\cW$ is separated.
Then because $\cW_0$ is defined over $\CC$ and is of finite type, to show that it is
separated we only need to verify the statement in Lemma \ref{valuative2}.  By Lemma \ref{valuative2},
$\cW_0$ is separated, therefore $\cW$ is separated.
\end{proof}

\begin{prop}\label{finite}
The stack $\cW^{-}$ is of finite type.
\end{prop}

This proposition will be proved in Subsection \ref{bound}.

\begin{prop}
The stack $\cW^{-}$ is proper.
\end{prop}

\begin{proof}
Assuming Proposition \ref{finite}, as argued in \cite[{\S  4}]{CLLL},  we only need to prove the statement in
Proposition \ref{proper-proof}. By the same proposition, we prove that $\cW^-$ is proper.
\end{proof}

\section{${G}$-fixed loci}
\label{Section4}

We fix a narrow $(g,\gamma,\bd)$ and an integer $N$. Let $G=(\CC\sta)^N$  act on $\cW$ via
\eqref{Gm}. 
We will characterize the fixed loci $\cW^\fix$.
As usual, a typical element $\xi\in\cW$ is 
\beq\label{nmsp}
\xi=(\sC,\Si^\sC,\sL,\sN,\varphi,\rho,\mu,\nu),\quad \varphi=(\varphi_i)_{i=1}^5,\ \mu=(\mu_i)_{i=1}^N.
\eeq
We will use $\rho=1$ to mean $\rho$ is nowhere vanishing, and the same applies to other sections.

\subsection{Associated graphs}
We associate a decorated graph to each $\xi\in \fixW$.
Recall that a graph $\Theta$ is determined by its vertices $V$, edges
$E$, legs $L$, and its flags
$F=\{(e,v)\in E \times V: v\ \text{incident to}\ e\}$. Our convention is that we will use $V(\ga)$ to emphasize the dependence of
$V$ on $\Theta$. If $\Theta$ is understood, we will use $V$ for simplicity.

Let $\xi\in\cW^\fix$ be of the form \eqref{nmsp}. Then ${G}$ has an induced action on $\sC$. We let $\sC\ufix\sub\sC$ be the closed subset fixed by ${G}$,
and let $\sC\umv=\sC-\sC\ufix$. Then  $\sC\ufix$ is proper and not
necessarily  of pure dimension one;  $\sC\umv$ is of pure dimension one with each
connected component smooth and { not proper}.


\begin{defi} \label{Ca}
Let $\xi\in\cW^\fix$. Let $\sC_0=\sC\cap (\mu=0)_\redd$,  $\sC_\infty=\sC\cap (\nu=0)_\redd$,
$\sC_1=\sC\cap (\rho=\varphi=0)_\redd$; let
 $\sC_{01}$ (resp. $\sC_{1\infty}$) be the union of irreducible components of
$\overline{\sC-\sC_0\cup\sC_1\cup\sC_\infty}$ in
$(\rho=0)$ (resp. in $(\varphi=0)$), and  $\sC_{0\infty}$ be the union of irreducible components of
$\sC$ not contained in $\sC_0\cup\sC_1\cup\sC_\infty\cup\sC_{01}\cup\sC_{1 \infty}$.
\end{defi}

We let $\Lambda=\{\infty,1,0\}$. Geometrically,  $\sC_{aa'}$ consists of curves connecting $\sC_a$ and $\sC_{a'}$,  and $\sC_a$
is the part of base curve  having level $a\in \Lambda$ defined later.

\begin{defi}\label{graph1}
Let $\xi\in \cW^\fix$. We associate to it a graph $\Theta_\xi$ as follows: 
\begin{enumerate}
\item (vertex) $V=V(\ga_\xi)$ is the set of connected components of $\sC\ufix$;

\item (edge) $E$ is the set of the connected components of $\sC\umv$;

\item (leg) $L\cong \{1,\cdots,\ell\}$ is the ordered set corresponding to the set of markings of $\Si^\sC$, $i \in L$ is attached to
$v\in V$ if $\Si_i^\sC\in \sC_v$;

\item (flag) $(e,v)\in F$ if $\sC_e\cap \sC_v\ne \emptyset$.
\end{enumerate}
We call $v\in V$ stable if $\sC_v$ is 1-dimensional; otherwise we call it unstable.
\end{defi}

The set of vertices has a partition $V=V_\infty\cup V_1\cup V_0$, and $v\in V_c$ if (the corresponding
component) $\sC_v\sub\sC_c$ for $c\in \Lambda$.
The set of edges $E$ has a partition $E=\cup_{c_1,c_2 \in \Lambda}E_{c_1 c_2}$ characterized by that
$e\in E$ lies in $E_{c_1 c_2}$ if 
the two ${G}$-fixed points of (the closure of the component corresponding to $e$) $\sC_e$ lie in $\sC_{c_1}$ and $\sC_{c_2}$ respectively.

We characterize the structure of $\xi|_{\sC_v}$ for a $v\in V$. In the remainder of this paper, $\alpha$ and $\beta$ are reserved
for integers in $[1,N]$.

\begin{lemm}\label{V} The followings characterize vertices in $V_\infty$, $V_1$ and $V_0$.
\begin{itemize}
\item[(a)] We have partitions $V_\infty=\cup_\alpha V_\infty^\alpha$, where $v\in V_\infty\ualp$ if the restrictions
$\rho|_{\sC_v}=1$,  $\mu\lalp|_{\sC_v}=1$,
$\mu_{\beta\ne\alpha}|_{\sC_v}=0$, and $\nu|_{\sC_v}=0$;
\item[(b)] We have particition $V_1=\cup_\alpha V_1^\alpha$, where $v\in V_1\ualp$ if the restrictions 
$\varphi|_{\sC_v}=\rho|_{\sC_v}=0$,  $\mu\lalp|_{\sC_v}=1$,
$\mu_{\beta\ne\alpha}|_{\sC_v}=0$, and $\nu|_{\sC_v}=1$;
\item[(c)] For $v\in V_0$, the restrictions 
$\varphi|_{\sC_v}$ is nowhere zero, $\mu|_{\sC_v}=0$, and $\nu|_{\sC_v}=1$.
\end{itemize}
\end{lemm}

The proof of Lemma \ref{V} is straightforward.

%

\begin{lemm}\label{E}
We have partitions $E_{1\infty}=\cup_{\alpha} E_{1\infty}\ualp$, $E_{01}=\cup_\alpha E_{01}^\alpha$.
Here $e\in E_{\bullet\bullet}\ualp$ if and only if
$e\in E_{\bullet\bullet}$, $\mu\lalp|_{\sC_e\cap \sC_1}\neq 0$, and $\mu_{\beta\ne\alpha}|_{\sC_e\cap\sC_1}= 0$.
\end{lemm}

\begin{proof}
We characterize the set $E_{1\infty}$. Let $\xi\in\cW^\fix$ and let $E$, etc., be the set of its edges, etc..
Let $e\in E_{1\infty}$. We let $p\in \sC_e\cap\sC_\infty$ and  $p'\in\sC_e\cap\sC_1$. We suppose $p$ is either a
marking or a node of $\sC$, thus $\omega_{\sC}^{\log}|_{p}=\bone_0$ is the trivial $G$-representation.
Then by the definition of NMSP fields, we have 
\begin{itemize}
\item[(a1):] $\nu|_p=0$, $\rho|_p\ne 0$, and $\mu|_p\ne 0$;
\item[(a2):] $\varphi|_{p'}=\rho|_{p'}=0$, $\mu|_{p'}\ne 0$, and $\nu|_{p'}\ne 0$.
\end{itemize}
By $\mu|_p\ne 0$ (resp. $\mu|_{p'}\ne 0$), there is an $\alpha$ (resp. $\beta$) such that $\mu\lalp|_p\ne 0$
(resp. $\mu\lbe|_{p'}\ne 0$). We claim that $\alpha=\beta$.

Suppose $\alpha\ne \beta$. We let $G\lalp\le G=(\CC\sta)^N$ be the $\alpha$-th factor of ${G}$,
  and $\bone_\alpha$ be the one-dimensional $G$-representation where $G_\alpha$ acts with weight one
and $G_{\beta\ne \alpha}$ acts trivially. 
Then by definition
$$\rho\in H^0(\sL^{-5}\otimes\omega_\sC^{\log}|_{\sC_e})^\fix,\quad \mu_\delta\in H^0((\sL\otimes\sN\otimes \bone_\delta)|_{\sC_e})^\fix,
\quad\nu\in H^0(\sN|_{ \sC_e})^\fix.
$$
Therefore, as $G$-representations, we have
\begin{itemize}
\item[(b1):] $\sL^{\otimes 5}|_{p}=\bone_0$, because $\rho|_{p}\ne 0$, $\rho$ is $G$-invariant, and
$\omega_{\sC}^{\log}|_{p}=\bone_0$.
\item[(b2):] $\sN|_{p}=-\bone\lalp$, because $\mu\lalp|_p\ne 0$, and $\mu\lalp\in H^0((\sL\otimes\sN\otimes \bone\lalp)|_{\sC_e})^{{ G}}$.
\item[(b3):] $\sN|_{p'}=\bone_0$, because $\nu|_{p'}\ne 0$.
\item[(b4):] $\sL|_{p'}=-\bone_\beta$, because $\mu\lbe|_{p'}\ne 0$, and
$\mu\lbe\in H^0((\sL\otimes\sN\otimes \bone_\beta)|_{\sC_e})^{{ G}}$.
\end{itemize}
Because $\alpha\ne\beta$, by looking at the $G\lalp$ action,  the conclusions (b2) and (b3) imply that the $G\lalp$ action on
$\sC_e$ is non-trivial. Then (b1) and (b4) imply that $\deg\sL|_{\sC_e}=0$. Since $p'$ is a scheme point, this forces $\sL^{\otimes 5}|_{\sC_e}=\sO_{\sC_e}$. Therefore, $\rho|_{\sC_e}$ is nowhere
vanishing, violating $\rho|_{p'}=0$. This proves that $\alpha\ne\beta$ is impossible.

We now suppose $p$ is a smooth non-marking of $\sC$. In this case, (a1) and (a2), and (b2)-(b4) still hold,
while (b1) must be replaced by
\begin{itemize}
\item[(b1'):] $\sL^{\otimes 5}\otimes \sO_{\sC_e}(-p)|_{p}=\bone_0$.
\end{itemize}
Combined with (b4), we conclude that $\deg\sL^{\otimes 5}=1$, which is impossible since $\sC_e\cong\Po$,
a (non-stack) scheme. This proves that $\alpha\ne\beta$ is impossible in this case, too.
Since $\alpha=\beta$, we deduce that $\mu\lalp|_{\sC_e\cap \sC_1}\neq 0$, and $\mu_{\beta\ne\alpha}|_{\sC_e\cap\sC_1}= 0$.

For the statement on $E_{01}$, we note that for any $e\in E_{01}$, with $p=\sC_1\cap\sE_e$ and $p'=\sC_0\cap\sC_e$,
we have $\varphi|_{p'}=0$ and $\mu|_{p}\ne 0$. Thus by the $G$-invariance, there is a unique $\alpha$ so that
$\mu\lalp|_{p}\ne 0$ and $\mu_{\beta\ne\alpha}|_p=0$. This proves the statement on $E_{01}$.
\end{proof}

\begin{rema}\label{remark1}
We remark that for $e\in E^\alpha_{1\infty}$, (b1)-(b4) (with $\alpha=\beta$) imply that $\sL\otimes\sN|_{\sC_e}\cong\sO_{\sC_e}$.
\end{rema}


\begin{lemm}
We have $E_{00}=\emptyset$; we have partitions $E_{\infty\infty}=\cup_{\alpha\ne \beta} E_{\infty\infty}\ualpbe$
and $E_{11}=\cup_{\alpha\ne \beta} E_{11}\ualpbe$, where $e\in E_{\bullet\bullet}\ualpbe$ if and only if
$\mu\lalp|_{\sC_e}\ne 0$, $\mu\lbe|_{\sC_e}\ne 0$, and $\mu_{\delta\ne\alpha,\beta}|_{\sC_e}= 0$.
\end{lemm}
\begin{proof}
The proof uses the similar arguments as in Lemma \ref{E}.
\end{proof}

Recall the convention we adopt throughout:
a vertex $v\in V$ is stable if $\dim \sC_v=1$; otherwise called unstable.
We let $V^S\subset V$ be the set of stable vertices.
For any $v\in V$, we let $E_v$ (resp. $L_v$)  be the set of edges (resp. legs) incident to $v$.
We introduce the convention that for any vertex $v\in V^S$ and $e\in E_v$, we denote by $q_{(e,v)}\defeq \sC_e\cap\sC_v$
the associated node in $\sC_v$.

\medskip
To add decorations to $\Theta_\xi$, we assign monodromies to flags $(e,v)\in F$. 

\begin{defi}\label{flag}Let $(e,v)\in F$.
\begin{enumerate}
\item In case $e\in E_{01}$, we assign $\gamma_{(e,v)}=(1,\rho)$;
\item in case $e\in E_{1\infty}$ and $v\in V_1$, we assign $\gamma_{(e,v)}=(1,\varphi)$;
\item in case $e\in E_{1\infty}$, $v\in V_\infty$ and $\deg\sL|_{\sC_e}=a+\frac{b}{5}$, where $b\in [1,5]$, we
assign $\gamma_{(e,v)}=\zeta_5^b$ when $b\ne 5$, $\gamma_{(e,v)}=(1,\varphi)$ when $b=5$;
\item in case $e\in E_{0\infty}$, we assign $\gamma_{(e,v)}=1$.
\end{enumerate}
\end{defi}

For (3), by \cite[Convention 2.1]{CLLL} and the convention before Proposition \ref{dim},
this choice of $b\in [1,5]$ makes $\gamma_{(e,v)}=b$ valid.

\begin{defi}\label{graph2} Let $\xi\in\cW^\fix$.
We endow $\Theta_\xi$ with the following decorations:
\begin{itemize}
\item[(a)] (genus) For $a\in V\cup E$, we let $g_a=h^1(\sO_{\sC_a})$;
\item[(b)] (degree) For $a\in V\cup E$, we let $(d_{0a},d_{\infty a})=
( \deg\sL\otimes\sN|_{\sC_a},\deg\sN|_{\sC_a})$; 
\item[(c)] (monodromy) The monodromy of the bundle $\sL$ for a leg is that of the marking it represents (recall that $\sL$ is representable);
\item[(c')] (monodromy) For $(e,v)\in F$,
we let $\gamma_{(e,v)}$ be as in Definition \ref{flag};
\item[(d)] (hour) For $v\in V_\bullet\ualp$, we let $\alpha_v=\alpha$, called the hour of $v$; for $e$ incident to $v$ we let $\alpha_{(e,v)}=\alpha_v$;
\item[(e)] (level) We say that elements of $V_\infty$ and $E_{\infty\infty}$ (resp. $V_1$ and $E_{11}$; resp. $V_0$)
have level $\infty$ (resp. level $1$; resp. level $0$). 
\end{itemize}
\end{defi}

We use $\Theta_\xi$ to emphasize that it is associated with $\xi$.
We caution that there may exist $\xi_1$ and $\xi_2$ lying in the same connected component of $\cW^\fix$ while
$\Theta_{\xi}\not\cong\Ga_{\xi_2}$.


\subsection{Balanced nodes and flat graphs}\label{flat}

We let $G\lggd$ be the set of all decorated graphs of $\xi\in\cW^\fix$, where $\cW=\cW\lggd$, modulo isomorphisms. We introduce the procedure  of
flattening the graph $\ga_\xi\in G\lggd$.

\begin{defi}
Let $\sC$ be a $\Gm$-twisted curve; 
let $q$ be a node of $\sC$, with $\hat \sC_1$ and $\hat \sC_2$ the two branches of
the formal completion of $\sC$ along $q$. We say that $q$ is $\Gm$-balanced 
if $T_q\hat \sC_1\otimes T_q\hat \sC_2\cong \bL_0$, the trivial $G$-representation.
\end{defi}

There are nodes of $\xi$ visible from $\Ga_\xi$. For one, any flag $(e,v)\in F$ with $v\in V^S$ associates to a node
$q_{(e,v)}$ that is $G$-unbalanced. The others are
\beq\label{NGa}
N(\Ga_\xi)=\{v\in V(\ga_\xi): \sC_v\text{ is a node of }\sC\}.
\eeq
Equivalently, $v\in N(\Ga_\xi)$ if $v$ is unstable, $|L_v|=0$, and $|E_v|=2$. For $v\in N(\Ga_\xi)$, we denote
by $q_v= \sC_v$, the associated node.
We call $v\in N(\Ga_\xi)$ $\Gm$-balanced if 
$q_v$ is a $\Gm$-balanced node.
We denote by $N(\Ga_\xi)^{\text{bal}}$ the set of ${G}$-balanced $v\in N(\Theta_\xi)$.
 We also denote $d_e:=\deg\sL|_{\sC_e}$ for convenience.

\begin{lemm}\label{unstable-q}
Let $v\in N(\Ga_\xi)$, and let $e$ and $e'$ be the two edges incident to $v$.
Then $q_v$ is ${G}$-balanced if and only if $v\in V_1$, $d_{e}+d_{e'}=0$, and
$(\sC_e\cup\sC_{e'})\cap\sC_\infty$ is a node or a marking of $\sC$.
\end{lemm}

\begin{proof} Let $v$, $e$ and $e'$ be as stated.
Suppose $v\in V_1(\ga_\xi)$ 
and $e\in E_{1\infty}$, then by Remark \ref{remark1} we have
$\mu\lalp|_{\sC_e}=1$, where the hour $h_v=\alpha$. Then the statement follows from the proof of \cite[Lem.\,2.14]{CLLL2}.

It is direct to argue that when $v\in V_\infty\cup V_0$, $v$ is not balanced. We omit the details here. This proves the
lemma.
\end{proof}



As in \cite[\S 2]{CLLL2}, when the statement of  Lemma \ref{unstable-q} holds, one necessarily has $e\in E_{1\infty}, e'\in E_{01}$  (or $e\in E_{01},e'\in E_{1\infty}$) because $d_{e}+d_{e'}=0$.
\smallskip

We now introduce the process of flattening a graph. 
We call a graph $\Ga\in G\lggd$ flat if $N(\Ga)^{\text{bal}}=\emptyset$.
For a graph $\ga\in G\lggd$ with $N(\Ga)^{\text{bal}}\ne\emptyset$, we define its flattening $\ga^\fl$, as follows.

\begin{cons}
For each $v\in N(\Ga)^{\text{bal}}$, which has $e\in E_{1 \infty}(\ga)$ and $e'\in E_{01}(\ga)$ incident to it,
we eliminate the vertex $v$ from $\Ga$, combine the two edges $e$ and $e'$ to
a single edge $\ti e$ incident to the other two vertices (in $V_\infty$ and $V_0$) that are incident to $e$ or $e'$,
and demand that $\ti e$ lies in $E_{0 \infty}$.  For decorations, we let $g_{\ti e}=0$, $d_{0\ti e}:=d_{0e'}$, $d_{\infty\ti e}=d_{\infty e}$  (using $d_e+d_{e'}=0$ we know $d_{0e'}=d_{\infty e}$), applying Definition \ref{flag} to assign $\gamma_{(\ti e, v)}$,
while keeping the remainder unchanged. We apply this procedure to every $v\in N(\ga)^{\mathrm{bal}}$ to obtain $\Ga^\fl$. 
\end{cons}

If $\Ga$ is flat, then $\Ga^\fl=\Ga$. We let $G\lggd^\fl=\{\Ga^\fl: \Ga\in G\lggd\}/\sim$.
We use $G\lggd^\fl$ to index an open and closed partition of $\cW^G$.

\medskip
Given a flat $\ga\in G\lggd^\fl$, we define a $\ga$-framed ${G}$-NMSP field to be a pair $(\xi,\eps)$, where $\eps:\ga_\xi^\fl\cong\Ga$
is an isomorphism (of decorated graphs). 
As in \cite{CLLL2},
we can make sense of families of $\Ga$-framed ${G}$-NMSP fields (cf. \cite[\S  2.4]{CLLL2}).
We then form the groupoid $\Wfix$ of $\Ga$-framed ${G}$-NMSP fields with obviously defined arrows;
$\Wfix$ is a DM stack, with a forgetful morphism
$$\iota_\ga: \Wfix\lra \fixW.
$$
Let $\cW_{(\Ga)}$ be the image of $\iota_\Ga$.
It is an open and closed substack of $\fixW$. The factored morphism
$\Wfix\to \cW\lrga$ 
is an $\Aut(\Ga)$-torsor. Further, we have an open and closed substack decomposition
\beq\label{decomp}
\cW^\fix=\coprod_{\ga\in G\lggd^\fl}\cW\lrga.
\eeq

The cosection localized virtual cycles
$[\cW_{(\ga)}]\virt$ are the terms appearing in the localization formula of $[\cW]\virt$. Because $\cW_\ga\to\cW_{(\ga)}$
is an $\Aut(\ga)$-torsor, the similarly defined virtual cycle $[\cW_{\ga}]\virt$ has (cf. \S  5; \cite[Coro.\,3.8]{CLLL2})
$$[\cW_{\ga}]\virt=|\Aut(\ga)|\cdot [\cW_{(\ga)}]\virt.
$$

\subsection{Proof of Proposition \ref{finite}}\label{bound}
We use the partition \eqref{decomp} to prove that $\cW^-$ is of finite type.
We first prove

\begin{lemm}\label{bd}
The stack $(\cW^{-})^\fix$ is of finite type.
\end{lemm}

The proof follows largely from that of \cite[\S  4.2]{CLLL}. The only necessary modification  is that,  for an MSP field
$\xi$ and for a $v\in V_\infty(\ga_\xi)$, we have $\sL\otimes\sN|_{\sC_v}\cong\sO_{\sC_v}$.
When $\xi$ is an NMSP field, and for a $v\in V_\infty(\Ga_\xi)$ and $e\in E_{\infty\infty}(\ga_\xi)$,
we still have $\sL\otimes\sN|_{\sC_v}\cong\sO_{\sC_v}$, but will have $\deg\sL\otimes\sN|_{\sC_e}>0$

As in \cite[\S  4.2]{CLLL}, we let $\Up_\xi$ be the dual graph of $\Si^\sC\sub \sC$. (Dual graphs have
vertices, edges, and legs, decorated by genus.)
As usual a vertex $v$ of a dual graph is {\sl stable} (resp. {\sl semistable}) if
$2g_v+n_v>2$ (resp. $\ge 2$), where $n_v=|L_v|+|E_v|$ and  $E_v$ (resp. $L_v$) is the set of edges (resp. legs) in
$\Up_\xi$ attached to $v$.
We prove.

\begin{lemm}\label{graph3}
The set $\Pi\defeq \{\Up_\xi\mid \xi\in (\cW^{-})^\fix(\CC)\}$ is a finite set.
\end{lemm}

\begin{proof} Our proof follows closely that of \cite[\S n 4.2]{CLLL}.  We will omit the reasoning if it is a mere repetition.

Let $\xi\in\cW^-\lggd(\CC)^\fix$. As the curve $\Si^\sC\sub \sC$ may not be stable, knowing the total genus
and the number of legs of its dual graph $\Up_\xi$ are not sufficient to bound the geometry of $\Up_\xi$. Our approach is
to use the information of line bundles $\sL$ and $\sN$ on $\sC$ given by $\xi$ to add legs to $\Up_\xi$ to
form a semistable $\ti\Up_\xi$

Let $\index=\{0,01,1,1\infty,\infty,0\infty\}$.
We continue to denote by $\sC_0$, $\sC_{01}$, etc., the substack of stacks in $\sC$
as defined in Definition \ref{Ca}. For $a\in \index$, we define
$V(\Up_\xi)_a=\{v\in V(\Up_\xi)\mid \sC_v\sub \sC_a\}$. As $\xi\in  (\cW^{-})^\fix$, $V(\Up_\xi)_{0\infty}=\emptyset$.

We then add auxiliary legs to vertices of $\Up_\xi$. For every $v\in V(\Up_\xi)$ we
add $3\deg\sL\otimes\sN|_{\sC_v}$ auxiliary legs to $v$. (Note that we always have $\deg\sL\otimes\sN|_{\sC_v}\ge 0$.)
The total new legs added is at most $3d_0$.

It is easy to show that,  with the added auxiliary legs,  all $v\in V(\Up_\xi)-V(\Up_\xi)_{1\infty}$ are {\sl stable}.
We now treat the vertices in $V(\Up_\xi)_{1\infty}$.
We let
$$E_\infty
=\{e\in E(\Up_\xi)\mid e\in E_v\ \text{for some}\ v\in V(\Up_\xi)_\infty\}.
$$
For $v\in V(\Up_\xi)_{1\infty}$, we define
\beq\label{del5}
\delta(v)=5\deg\sN|_{\sC_v}-|E_v\cap E_{\infty}|\in \ZZ_{\ge 0}.
\eeq
Since $|E_v\cap E_{\infty}|=0$ or $1$, and since $\deg\sN|_{\sC_v}\ge \frac{1}{5}$, $\delta(v)$ takes
value in $\ZZ_{\ge 0}$ because $\deg\sN|_{\sC_v}\in \frac{1}{5}\ZZ$.
To each $v\in V(\Up_\xi)_{1\infty}$,  we add $2\delta(v)$ many auxiliary  legs to $v$.

We show that 
the number of new legs added to $V(\Up_\xi)_{1\infty}$ is bounded by $10d_\infty+4g+2\ell$.
Let $\sD_1, \ldots, \sD_s$ be the irreducible components of $\sC_\infty$, and  $\ell_i$ be the number
of markings on $\sD_i$. Because $\rho|_{\sD_i}$ is nowhere vanishing, and using $\deg \sL\otimes\sN|_{\sD_i}\ge 0$,
we have
$$0=-5\deg\sL|_{\sD_i}+\deg\omega_{\sC}^{\log}|_{\sD_i}\le
5\deg\sN|_{\sD_i}+\bl2g(\sD_i)-2+|\sC_{\text{node}}\cap\sD_i|+\ell_i\br.
$$
Let $\sC_\infty^1,\cdots\sC_\infty^t$ be the connected components of $\sC_\infty$.
Using $\cup\sD_i=\cup \sC_\infty^j$, we get
$$\sum_{i=1}^s\big(2-2g(\sD_i)-|\sC_{\text{node}}\cap\sC_\infty|\big)=
\sum_{j=1}^t \big(2-2g(\sC_\infty^j)\big)-\sum_{v\in V(\Up_\xi)_{1\infty}}|\sC_{\text{node}}\cap \sC_v|.
$$
Using $\deg\sN|_{\sC_v}=0$ unless $v\in V(\Up_\xi)_{1\infty}\cup V(\Up_\xi)_\infty$, letting
$\ell_\infty=\sum_i \ell_i$, we have
$$d_\infty \ge \deg\sN=\sum_{j=1}^t \frac{1}{5}\bl2-2g(\sC_\infty^j) -|\sC_\infty^j\cap\sC_{\text{node}}|\br -\frac{\ell_\infty}{5}+
\sum_{v\in V(\Up_\xi)_{1\infty}}\deg \sN|_{\sC_v}.
$$
Adding
$\displaystyle \sum_{v\in V(\Up_\xi)_{1\infty}} |E_v\cap E_{\infty}|=\sum_{j=1}^t |\sC_\infty^j\cap\sC_{1\infty}|,
$
we obtain
$$d_\infty  \ge \frac{2 t}{5}- \sum_{j=1}^t\frac{2}{5}\cdot g(\sC_\infty^i)-\frac{\ell_\infty}{5}+
\frac{1}{5} \sum_{v\in V(\Up_\xi)_{1\infty}} \delta(v).\quad \ \,
$$
Thus the total number of auxiliary legs added to vertices in $V(\Up_\xi)_{1\infty}$, which is
$ \sum 2\delta(v)$,  is 
bounded by $10d_\infty+4g+2\ell$; the number $t$ of connected components of $\sC_\infty$ is bounded by the same number, too.

Let $\ti\Up_\xi$ be the resulting graph after adding auxiliary legs to $v\in V(\Up_\xi)$ according
to the rules specified above. As in \cite[\S  4.2]{CLLL}, one shows that every vertex of $\ti \Up_\xi$ is
semistable, and that $\ti\Up_\xi$ contains no chain of strictly semistable vertices of length more than two.

Finally, we let $(\ti\Up_\xi)\st$ be the stabilization of $\ti\Up_\xi$.
Since the genus of $(\ti\Up_\xi)\st$ is $g$ and the
number of markings of $(\ti\Up_\xi)\st$ is bounded by a constant $K$,
$$\Pi\st=\{(\ti\Up_\xi)\st\mid {\Up_\xi}\in\Pi\}
$$
is finite. Since  the contraction $\ti\Up_\xi \rightsquigarrow (\ti\Up_\xi)\st$ is by contracting chains of at most length two
strictly semistable vertices, $\Pi$ is finite.
%
\end{proof}

The proof of Lemma \ref{bd} follows from the finiteness of Lemma \ref{graph3}.

\begin{proof}[Proof of Proposition \ref{finite}]
We let
$\cZ=\{\xi\in\cW^-(\CC)\mid \lim_{t=(t_1,\cdots,t_N)\to 0} t\cdot \xi\in (\cW^-)^\fix\}$. Since
$(\cW^-)^\fix$ is of finite type, $\cZ$ is of finite type. By Proposition \ref{proper-proof},  $\cZ=\cW^-$.
This proves the Proposition.
\end{proof}

%
%


\section{Irregular graphs and their associated vanishings}
In this section, we continue to work with a narrow $(g,\gamma,\bd)$, and abbreviate $\cW=\cW\lggd$.
We introduce the notion of regular graphs. We prove that the virtual localization contribution associated to  an irregular graph vanishes.

\subsection{Regular graphs} Let $\Ga\in G\lggd^\fl$ and $v\in V^S_\infty(\Ga)$. We let
$$\gamma_v=\{\gamma_l: l\in L_v\}\cup \{\gamma_{(e,v)}: e\in E_v\}.
$$
%
Following the convention adopted before Proposition \ref{dim}, by abuse of  notation, we also use
$\gamma_a=m_a\in [0,4]$ if $\gamma_a=\zeta_5^{m_a}$; $\gamma_a=5$ if $\gamma_a=(1,\varphi)$. 
We often write $\gamma_v=(0^{k_0}\cdots 5^{k_5})$.

\begin{defi}
We call a vertex $v\in V_\infty^S(\ga)$ (or $(g_v,\gamma_v)$) exceptional if (i) $g_v=0$ and (ii) either $\gamma_v=(1^{2+k}4)$,
or $\gamma_v=(1^{1+k}23)$, or $\gamma_v=(0^21^{1+k})$. 
\end{defi}

We remark that since $\gamma$ is narrow, $\gamma_v=(0^21^{1+k})$ occur only if the $(0^2)$ are associated to two nodes
of $\sC_v$.

\begin{defi}\label{de-regu}
We call a vertex $v\in V_\infty(\Ga)$ regular if the   followings hold:
\begin{enumerate}
\item In case $v$ is stable, then either $v$ is exceptional, or $\gamma_v=(1^{k_1}2^{k_2})$.
\item In case $v$ is unstable attached to an edge in $E_{1\infty}$, then $\sC_v$ is a non-scheme marking.
\end{enumerate}
We call $\Ga$ regular if it is flat, all its vertices $v\in V_\infty(\ga)$ are regular, and $E_{0\infty}(\ga)=\emptyset$.
We call $\ga$ irregular if it is not regular.
\end{defi}

\begin{defi}\label{sim}
Given two cycle classes $A$ and $A'\in A\lsta^G\cW^-$, we say $A\sim A'$ if there is a proper $G$-invariant substack
$W'\sub \cW$ containing $\cW^-$, letting $\eta: \cW^-\to W'$ be the inclusion, such that $\eta\lsta  A=\eta\lsta A'$.
\end{defi}

The vanishing result we will prove is Theorem \ref{thm3}, restated here.

\begin{theo} \label{main}
Let $\ga\in G\lggd^\fl$. Suppose $\ga$ is irregular, then 
$[\cW_{\ga}]\virt\sim 0$.
\end{theo}

The proof of this theorem will occupy the remainder of this section.
As we will see, our proof follows closely the proof for the case of
MSP fields in \cite{IG}.

\subsection{The trimming of graphs}\label{5.2}
Let $\Theta$ be as in Theorem \ref{main}. As in the case of MSP fields, we first trim all edges and vertices of $\Theta$ in
$E_{01}\cup E_{1\infty}\cup V_1$, resulting a graph $\Theta'$ with $V_1(\Theta')=\emptyset$. We then trim
all legs decorated with $\zeta_5$, and some others to be specified shortly for simplification of  our discussion.

We begin with the issue of various constructions of the virtual cycles. 
We recall that the virtual cycle $[\cW_\ga]\virt$ is constructed using the relative perfect
obstruction theory $\phi_{\cW_\ga/\cD_G}$, relative to the stack of tuples of pointed twisted curves with invertible sheaves
(cf. \cite[\S 3]{IG}). 
By \cite{CLLL2}, it equals the cycle constructed using the
$G$-invariant part $(\phi_\cW)^\fix$ of the absolute obstruction theory, as the later is what appears in
the virtual localization formula. In the following, we
will work with the obstruction theory $\phi_{\cW_\ga/\cD_G}$, knowing that it gives the desired class.

It is convenient to introduce the notion of webs. 
Given a $\Ga\in G\lggd^\fl$, we let $\Ga_\infty$ be the graph (possibly disconnected) obtained as follows:
we remove $E_{01}$($=E_{01}(\ga)$) and $V_1\cup V_0$ from $\ga$, and remove all
legs attached to $V_1\cup V_0$; we keep all legs attached to $V_\infty$, and replace every
$e\in E_{1\infty}\cup E_{0\infty}$ by a new leg $l_e$ attached to the vertex $v\in V_\infty$ to which $e$ is attached.

For the decoration, we keep the decorations of $V_\infty(\Ga_\infty)$ and $E_{\infty\infty}(\Ga_\infty)$; for a leg $l\in
L(\Ga_\infty)$, we keep its decoration in case $l$ is from a leg in $\Ga$; in case $l=l_e$ is from
(replacing) an edge $e\in E_{1\infty}(\ga)$, letting $v\in V_\infty(\Ga)$ be the vertex to which $e$ is attached and letting
$\zeta_5^{5-a}$ ($a\in [0,4]$) be the monodromy of $\sL|_{\sC_v}$ along the node
$q_{(e,v)}=\sC_v\cap \sC_e$. If $a\in [1,4]$  we decorate $l$ by $\gamma_{}\defeq \zeta_5^{5-a}$;
if $a=0$ we decorated $l$ by $(1,\varphi)$.
In case $l=l_e$ is from an edge in $E_{0\infty}(\ga)$, we set $\gamma_{l}\defeq 1$, a broad marking.
Note that it is possible that some of the connected components of $\Theta_\infty$ is degenerate, namely those
connected components that has no stable vertex and no edges.
We call a connected non-degenerate component of $\Ga_\infty$ a web of $\ga$ at infinity.

Similarly, for any $\Ga\in G\lggd^\fl$, we can form a graph $\Theta'$ resulting from trimming all edges in
$E_{01}(\Theta)\cup E_{1\infty}(\Theta)$, removing all vertices in $V_1(\Theta)$, and discard all degenerate connected components.
Like before, when we trim an
edge $e\in   E_{1\infty}$ from $\Theta$, as we are removing all vertices in $V_1$ at the same time, we only need to replace
$e$ by a leg $l_e$ attached to the vertex $v\in V_\infty(\Theta)$ of which the $e$ was attached; the decoration is the same as
we did in forming the webs of $\Theta$.
Trimming edges in $E_{01}$ is similar; we add a leg decorated by $(1,\rho)$ to $v$ every time we remove an edge from
$v\in V_0$.

\begin{lemm}\label{trim-1}
Let $\Ga\in G\lggd^\fl$ be as before, and let $\Theta'$ be the result after trimming all edges and vertices in
$E_{01}(\Theta)\cup E_{1\infty}(\Theta)\cup V_1(\Theta)$. Then $[\cW_\Theta]\virt=0$ if
$[\cW_{\Theta'}]\virt=0$.
\end{lemm}

\begin{proof}
The proof is a word by word repetition of \cite[Proof of Proposition 4.1]{IG}, and will be ommitted.
\end{proof}

We now consider a graph $\Ga\in G\lggd^\fl$ with $E_{01}\cup E_{1\infty}\cup V_1=\emptyset$. To mimic the proof of
\cite[Lemma 4.3]{IG}, we need to do trimming of $\zeta_5$ legs, and others.

\begin{defi}\label{neutral}
Let $\Theta$ be as before.
\begin{itemize}
\item[i.] Let $e\in E_{\infty\infty}(\Theta)$, and let $v_-$ and $v_+$ be its two vertices.
Suppose $q_{(e,v_+)}=\sC_e\cap\sC_{v_+}$ is a node in $\sC$.
We say $e$ is of marking (resp. leaf-end) type if $g_{v_-}=0$, and $|E_{v_-}|=|L_{v_-}|=1$ (resp.
$g_{v_-}=|L_{v_-}|=0$, and $|E_{v_-}|=1$.) We say $e$ is of nodal type if
$q_{(e,v_-)}$ is a node in $\sC$.
\item[ii.] We say $l$ in $L(\Theta)$ is spare if $l$ is decorated by $\zeta_5$
and is attached to a vertex $v\in V(\Theta)$ so that $2g_v+n_v>3$.
\item[iii.] We say $v\in V(\Theta)$ is neutral if $g_v=0$, both $E_v$ and $L_v$ are non-empty,
$|E_v|+|L_v|=3$, and at least one $l \in L_v$ is a $\zeta_5$-marking.
\end{itemize}
\end{defi}

One simplification we will make is to trim all edges of
marking or leaf-end type. Let $e\in E(\Theta)$ be of marking type, with $q_{(e,v_+)}$ a node of $\sC$,
and let $l$ be the leg incident to $v_-$.
We construct a new web $\fa_e$ by removing $e$ and $v_-$ from $\fa$, and then attaching $l$ to $v_+$ with
decoration unchanged.

In case $e$ is of leaf-end type, $\deg \cL|_{\cC_e}=-1/5$, thus the (only) node $q_{(e,v^+)}$ on $\sC_e$ has
$\cL$-monodromy $\zeta_5^4$; 
we let $\fa_e$ be the web obtained from removing $e$ and $v_-$ from $\fa$
and attaching a new $\zeta_5$-decorated leg $l$ to $v_+$. We call $\fa_e$ the trimming   of $\fa$ by its
edge of marking type or leaf-end type, respectively.

\begin{lemm}\label{marking} Let $e$ be  of marking or   leaf-end type. Let $\fa_e$ be as constructed above. Then
there is a canonical morphism
$\psi_e: \cW_\fa\lra \cW_{\fa_e}$, which is a gerbe banded by $\bmu_\delta$, where $\delta=d_{0e}$, such that
\beq\label{12}
[\cW_\fa]\virt=(\psi_e)\sta[\cW_{\fa_e}]\virt.
\eeq
\end{lemm}

\begin{proof}
Let $e\in E(\fa)$ be an edge of leaf-end type. Let
\beq\label{xi2}
\xi=(\sC,\Sigma^\sC, \sL, \sN, \varphi,\rho,\mu)\in \cW_{\fa}(\CC).
\eeq
We will construct a
datum $\ti\xi\in \cW_{\fa_e}(\CC)$ that provides the morphism $\psi_e$ as stated.
Let $\sC_e\sub\sC$ be the rational curve associated with $e$;
let $q_-\in \sC_e$ be the leaf-end point of $\sC_e$ fixed by $G$.
Since $\sL^{\otimes 5}|_{\sC_e}\cong \omega^{\log}_\sC|_{\sC_e}$,  we have $\deg\sL|_{\sC_e}=-1/5$. As
$q_+=q_{(e,v_+)}$ is the only stacky point on $\sC_e$, the monodromy of $\sL|_{\sC_e}$ along $q_+$ is
$\zeta_5^4$.
Let $=\sC'$ be the (sub)curve that is the closure of $\sC-\sC_e$ in $\sC$. The monodromy of
$\sL|_{\sC'}$ along $q_+$ is $\zeta_5$.

We let $\ti\sC=\sC'$. We declare $q_{(e,v_+)}\sub\ti\sC$ to be a new marking, decorated by $\zeta_5$;
we set $\ti\sL=\sL|_{\sC'}$, and let $\ti\rho:\ti\sL^{\otimes 5}\cong\omega_{\ti\sC}^{\log}$
be the restriction of $\rho$ to $\sC'$.
For the other data, we set $\ti\mu=\mu|_{\sC'}$, and set $\ti\varphi=\varphi|_{\sC'}$.
Then
\beq\label{tixi2}
\ti\xi=(\ti\sC,\Sigma^{\ti\sC}, \ti\sL,\ti\sN, \ti\varphi,\ti\rho,\ti\mu)\in \cW_{\fa_e}(\CC).
\eeq
Clearly, the family version of this construction provides the desired morphism $\psi_e$.

We show that $\psi_e$ is a gerbe banded by $\bmu_\delta$. Indeed,  any $t\in\CC\sta$ acting nontrivially on $\sC_e$ fixing both
$q_-$ and $q_+$ and trivially on $\sC'\sub\sC$ induces an arrow
$\xi \to \xi^t$ in $\cW_\fa$. Since $\delta=\deg\sL\otimes\sN|_{\sC_e}$, $\xi=\xi^t$ if and only if $t\in\bmu_\delta$.
This proves that $\psi_e$ is a gerbe.

Finally, a repetition of the argument in \cite[Theorem 4.10]{CLL} shows that the relative obstruction theory of
$\cW_\fa$ and $\cW_{\fa_e}$ are identical under the morphism $\psi_e$. Thus the identity
\eqref{12} follows.

The case where $e$ is of marking type is similar, and will be omitted. This proves the Lemma.
\end{proof}

When $l$ is a spare marking, after removing $l$ from $\Theta$,  the resulting graph $\Theta'$ remains a graph of
NMSP fields.

\begin{lemm}\label{trim-5}
Let $\Theta'$ be the resulting graph after removing a spare leg $l$ from $\Theta$. Then there is a
canonical one-dimensional morphism $\psi: \cW_\Theta\to\cW_{\Theta'}$ so that
$$[\cW_\Theta]\virt=\psi\sta[\cW_{\Theta'}]\virt.
$$
\end{lemm}

\begin{proof}
We omit the proof since it is parallel to that in \cite[Theorem 4.10]{CLL}.
\end{proof}

Let $v\in V(\Theta)$ be a neutral vertex. There are two cases: one is when $v$ has one edge $e$ and two legs, $l_1$
decorated by $\zeta_5$, and $l_2$ decorated by $\zeta_5^b$. We let $\Theta'$ be the graph obtained by
removing $v$ from $\Theta$, replacing the edge $e\in E_v$ by a new leg $l$ decorated by $\zeta_5^b$.

\begin{lemm}\label{trim-6}
Let the situation be as stated above, and let $\Theta'$ be the resulting graph after removing the neutral vertex $v$ from $\Theta$, as constructed.
Then there is a canonical morphism $\psi: \cW_\Theta\to\cW_{\Theta'}$ so that the conclusion of Lemma \ref{trim-5} holds.
\end{lemm}

\begin{proof}
We continue to use the notation introduced before the statement of the lemma. For any $\xi\in\cW_\Theta$,
$$\xi=(\sC,\Sigma, \sL,\sN, \rho,\varphi,\mu,\nu),
$$
we let $\sC'$ be obtained from $\sC$ with $\sC_e$ removed; let $\sL'=\sL|_{\sC'}$, $\Sigma'=\Sigma\cup\{q_{(e,v)}\}$, $\rho'=\rho|_{\sC'}$, etc..
Since $g_v=0$, the monodromy of $\sL|_{\sC_v}$ along its only node is $\zeta_5^{-b}$.
Then
$$\xi'=(\sC',\Sigma', \sL',\sN', \rho',\varphi',\mu',\nu')\in \cW_{\Theta'}.
$$
This construction defines a morphism $\psi: \cW_\Theta\to\cW_{\Theta'}$.
Parallel to the argument in \cite[Theorem 4.10]{CLL}, we easily see that the conclusion of Lemma \ref{trim-5} hold.
\end{proof}

The other case for a neutral $v\in V(\Theta)$ is when $v$ has two edges $e_1$ and $e_2$, and one legs $l$,
decorated by $\zeta_5$. In this case, we let $\Theta'$ be the graph obtained by removing $l$ and
make $v$ unstable.

\begin{lemm}\label{trim-7}
Let the situation be as stated above, and let $\Theta'$ be the resulting graph after removing the leg $l$, as constructed.
Then there is a canonical morphism $\psi: \cW_\Theta\to\cW_{\Theta'}$ so that the conclusion of Lemma \ref{trim-5} holds.
\end{lemm}

\begin{proof}
The proof is similar, using that the monodromy of $\sL|_{\sC_v}$ along the two nodes of $\sC_v$ are inverse to each other.
\end{proof}

We first prove a special case of the
desired vanishing Theorem \ref{main}.
Recall a {\sl string} of $\Ga$ is an $e\in E_{0\infty}(\Ga)$ so that the vertex $v$ of $e$
lying in $V_0(\Ga)$ is unstable and has no other edge attached to it.

\begin{prop}\label{reduction1}
Let $\Ga\in G\lggd^\fl$ be irregular and not a pure loop (see after Theorem \ref{thm3}). Suppose it does not contain any string, then
$[\cW_\Ga]\virt=0$.
\end{prop}

\begin{proof}
Our proof follows exactly the proof of \cite[Proposition 4.1]{IG}. First applying Lemma \ref{trim-1}, we
can trim all edges and vertices of $\Theta$ in $E_{01}\cup E_{1\infty}\cup V_1$,
and reduce the proof to that when $\Theta$ has no edges in $E_{01}\cup E_{1\infty}$.
This way, we are reduced to prove the analogous of \cite[Lemma 4.3]{IG}.

In the proof of \cite[Lemma 4.3]{IG}, one first trim the graph $\Theta$ so that there is no leg decorated  by $\zeta_5$.
In our case, we first repeatedly trim all edges of marking type and leaf-end type, trim all spare legs, and trim all
neutral vertices, resulting a graph $\Theta'$ with no marking type and leaf-end type edges, no spare legs,
and no neutral vertices. We remark that trimming a neutral vertex may create a new marking type edge,
thus this trimming may have to be repeated
as necessary.
By Lemma \ref{trim-5}
and \ref{trim-6}, $[\cW_\Theta]\virt=0$ if $[\cW_{\Theta'}]\virt=0$. Thus we are reduced to prove the Proposition when $\Theta$ has no spare legs and no neutral vertices.

Suppose $\Theta$ has no spare legs and no neutral vertices. We claim that, in case $[\cW_\Theta]\virt\ne 0$, the graph $\Theta$ has no legs decorated by $\zeta_5$. Indeed, by no spare legs and no neutral vertices assumption,
any  leg $l$  decorated by $\zeta_5$  of $\Theta$ must be attached to an unstable vertex $v$ that has an edge
$e\in E_{0\infty}(\Theta)$ attached to it. As $[\cW_\Theta]\virt\ne 0$, then $\cW_\Theta^-\ne\emptyset$, and then
$e$ must be the flattening of a pair of edges from $E_{01}$ and $ E_{1\infty}$.
By Lemma \ref{unstable-q} $\cC_v$ must be a scheme
point of $\sC$, contradicting to that $l$ is decorated by $\zeta_5$.

This way, to prove the proposition, it suffices to consider the case where $\Theta$ has no vertices in
$V_1$ and no legs decorated by $\zeta_5$, which is the result of the following Lemma. This proves the
proposition.
\end{proof}

\begin{lemm}
Let $\Theta$ be as in Proposition \ref{reduction1}; assume further that it has no vertices in $V_1$
and no legs decorated by $\zeta_5$. Then $\expdim \cW_\Theta<0$.
\end{lemm}

\begin{proof}Let $\Theta$ be as in the lemma.
Let $\fa_1,\cdots,\fa_k$ be the connected components of $\Theta_\infty$. We first assume that all
$\fa_i$ are one-vertex (plus legs) graphs. We then let $\Theta^{\text{end}}$ be obtained from  $\Theta$ after deleting all ``hour"
decorations of $\Theta$. A moment of thought shows that $\Theta^{\text{end}}$ is a (flat) graph of MSP fields. Furthermore,
by going through the definition of obstruction theory, we see that
\beq\label{eq}
\expdim \cW_\Theta=\expdim \cW_{\Theta^{\text{end}}}.
\eeq
As $\Theta^{\text{end}}$ is a flat graph of MSP fields satisfying the
assumption of Proposition \ref{reduction1}, by \cite[]{IG}, we get
\beq\label{neg}\expdim \cW_{\Theta^{\text{end}}}<0.
\eeq
Combined with \eqref{eq}, we prove the lemma in this case.

In general, let $\fa_i$ be such that it is not a single vertex graph. We let $\fa_i'$ be a single vertex graph with
vertex $v_i$, of genus the total genus of $\fa_i$;
we let the legs (resp. edges) incident to $\fa_i'$ be the collection of all legs (resp. edges in $E_{0\infty}(\Theta)$)
incident to $\fa_i$,
with their decorations unaltered except all decorations of ``hours" are deleted. The assigning of degree $d_0$ and
$d_\infty$ to $\fa_i'$ requires a little care. As we know, there is a unique choice of degrees $d_{0v_i}$ and
$d_{\infty v_i}$ making $\fa_i'$ a single vertex graph of MSP (namely, NMSP with $N=1$) fields lying at level
$\infty$; the choice depends on
the genus $g_{v_i}$, the decorations of the legs,  and edges $L_{v_i}\cup E_{v_i}$ incident to $v_i$.
We call this the {\sl MSP choice of} $d_{0v_i}$ and $d_{\infty v_i}$.

We do a parallel construction to $\Theta$: for any connected component $\fa_i$ of $\Theta_\infty$,
we (after removing all edges in $\fa_i$) combine all vertices in $\fa_i$ into a single vertex $v_i$;
we make all edges in $E_{0\infty}$ (resp legs) in $\Theta$ that are incident to ``$\fa_i$"
to incident to the vertex $v_i$;
we denote the resulting graph to be $\Theta^{\text{end}}$.
We decorate $\Theta^{\text{end}}$, making it a graph of MSP fields,
according to the following rule: we know
$V_0(\Theta^{\text{end}})=V_0(\Theta)$, 
$E(\Theta^{\text{end}})=E_{0\infty}(\Theta)$ and $L(\Theta^{\text{end}})=L(\Theta)$;
we keep the decorations of the edges and legs of $\Theta^{\text{end}}$ unchanged, and keep the decorations of the vertices $V_0(\Theta)$ unchanged, except we delete all ``hour" decorations
from all of them.
For the vertex $v_i$ from $\fa_i$, we let $g_{v_i}$ be
the total genus of $\fa_i$; we let $d_{0 v_i}$ and $d_{\infty v_i}$ be the {\sl MSP choice of}
$d_{0v_i}$ and $d_{\infty v_i}$,  specified at the end of the previous paragraph.

We now prove the lemma. First, $\Theta^{\text{end}}$ is a graph of MSP fields satisfying the assumption of
Proposition \ref{reduction1}, by \cite[Proposition 4.1]{IG}, we have \eqref{neg}.
We then apply the argument of the proof of \cite[Theorem 4.7]{CLL}
to conclude
\beq\label{MSP-deg}
\expdim \cW_{\Theta} \le\expdim \cW_{\Theta^{\text{end}}}.
\eeq
Indeed, it is direct to check that if every $\fa_i$ has at most one stable vertex, then the equality in \eqref{MSP-deg}
holds; otherwise, the strict inequality holds.

In any case, we get $\expdim \cW_{\Theta}<0$.
\end{proof}

\subsection{The proof of the vanishing}
The challenging part is the case where $\ga$ does contain a string, where in this case
the expected dimension of $\cW_\ga$ could be non-negative, thus the vanishing result does not follow from the
expected dimension calculation. 

\begin{defi}\label{-string}
Let $\Ga\in G\lggd^\fl$ and let $e\in E_{0\infty}(\Ga)$ be a string (thus a leaf edge).
Let $v_-\in V_0(\Ga)$ and $v_+\in V_\infty(\Ga)$ be its vertices.
The trimming of $e$ from $\Ga$ is by first removing $e$,  $v_-$ and all legs attached to $v_-$, and
then attaching a leg, called the distinguished leg, decorated by $(1,\varphi)$ to $v_+$.
\end{defi}

We fix a string $e$ of $\ga$  and  let $\ga'$ be the result of trimming $e$ from $\ga$. We form $\cW_{\ga'}$.
As in the discussion before \cite[(5.12)]{IG}, we get a morphism
$(\cW_\ga)\lred\to \cW_{\ga'}$. We define
$$\cW_\ga^\sim=(\cW_\ga)\lred\times_{\cW_{\ga'}}\cW_{\ga'}^-.
$$
Clearly, $\cW_\ga^-\sub \cW_\ga^\sim$. We let $\ti\kappa: \cW_\ga^{\sim}\to \cW^-_{\ga'}$ be the induced projection,
and let $\jmath: \cW_\ga^{-}\to \cW_\ga^\sim$ be the inclusion.

\begin{prop}\label{reduction} The morphism $\ti\kappa$ is proper, thus $\cW_\ga^\sim$ is proper.
Then there is a rational $c\in\QQ$ such that
$$\jmath\lsta[\cW_\ga]\virt=c\cdot \ti\kappa\sta [\cW_{\ga'}]\virt\in A\lsta(\cW^\sim_{\ga}).
$$
\end{prop}

\begin{proof}
The proof follows exactly as that of \cite[Proposition 5.5]{IG}.
\end{proof}

\begin{proof}[Proof of Theorem \ref{main}]
The proof is the same as that at the end of \cite[\S  5]{IG}.
\end{proof}

\section{Localization formula and separation of nodes}
\def\Mbar{\overline{\cM}} \def\alp{\alpha}
\def\lophi{_{(1,\varphi)}}
\def\lorho{_{(1,\rho)}}

We state the localization formula for the moduli of NMSP fields. As its proof is parallel to that in \cite{CLLL2}, we will
not repeat them here.
Let $(g,\gamma,\bd)$ be narrow, and let $\cW=\cW\lggd$ be the moduli of stable NMSP fields of data $(g,\gamma,\bd)$.
Let $\Theta\in G\lggd^\reg$ be regular, with
$$\cW_\Theta\lra \cW_{(\Theta)}\sub \cW^G
$$
being the stack of $\Theta$-framed NMSP fields, and $\cW_{(\ga)}\sub\cW^G$ its image stack, where the later is both open
and closed. 
For each vertex $v\in V_0$($=V_0(\Theta)$), we let $[v]=\{v\}\cup E_v$, where $E_v$ is the set of edges incident to $v$.
We let $\fa_1,\cdots,\fa_s$ be the webs of $\Theta$ at infinity.

We have an isomorphism
$$
\cW_{\ga}\cong \prod_{v\in V_0}{\cW_{[v]}}
\times \prod_{v\in V^S_1}\cW_{v} \times \prod_{\fa\in\{\fa_1,\cdots,\fa_s\}}\cW_{\fa}
\times \prod_{e\in E_{1\infty}} \cW_e
\times \prod_{e\in E_{11}} \cW_e.
$$
Here, the notations for $\cW_a$, except $\cW_\fa$, are the same as in \cite{CLLL2}; $\cW_\fa$ is as in \S \ref{5.2}.
We comment that every edge in $E_{01}$ is contained in one $[v]$.

\subsection{The fixed parts}
As before, we denote by $\cW_\Theta^-$ the degeneracy locus of the cosection used to construct the cosection
localized virtual cycle, which is $\cW_\Theta^-=\cW_\ga\cap\cW^-$, and is proper.
Let $\iota_\ga:\cW_\Theta^-\sub\cW_\ga$ be the inclusion.

We let the automorphism group of a point in $\cW_e$ be $G_e$.
For instance, when $e\in E_{01}(\ga)$, we have $G_e=\bmu_{d_e}$, and $ \cW_e\cong \sqrt[d_e]{\cO_{\PP^4}(1)/\PP^4}$.
Here $d_e=d_{0e}-d_{\infty e}=\deg(\sL|_{\sC_e})$.
Let
$$G_E=\prod_{e\in E(\Theta)-E_{\infty\infty}(\Theta)}G_e.
$$
Then the induced morphism
$$\cW_\ga\lra \prod_{v\in V_0\cup  V_1^S} \cW_v\times  \prod_{\fa\in\{\fa_1,\cdots,\fa_s\}}\cW_\fa
$$
is a $G_E$-gerbe.


\begin{proposition}\label{prop-loc}
Let $\Ga\in G\lggd^\reg$ be regular. Then
$$ [\cW_{(\Ga)}]\virt= \frac{1}{|\Aut(\Ga)|} \frac{1}{|G_E|}
(\imath_{\ga})_\ast\biggl( \prod_{v\in V_0\cup  V_1^S} [\cW_v]\virt\times  \prod_{\fa\in\{\fa_1,\cdots,\fa_s\}}[\cW_\fa]\virt\biggr).
$$
\end{proposition}

We now study each individual term $[\cW_v]\virt$ appearing in the formula.
As before, for $v\in V$ we denote by
$S_v$ to be the set of markings $\Sigma_i^\sC\in \sC_v$, and denote by $E_v$ to be the set of edges incident to $v$.
Let $Q_5\sub \Pf$ be the Fermat quintic.

\begin{lemm}The cosection localized virtual cycles of the fixed part are
\begin{enumerate}
\item
For $v\in V_1^S$, $\cW_v\cong \overline{\cM}_{g_v,|E_v\cup S_v|}$, and $[\cW_v]\virt=[\cW_v]$;
\item
for $v\in V_0^S$, $[\cW_{v}]\virt=
(-1)^{d_v+1-g}[\Mbar_{g_v, E_v\cup S_v}(Q_5,d_v)]\virt $;
\item
for $v\in V^U_0$, or $v\in V_0^S$ with $g_v=d_v=0$,
$[\cW_{v}]\virt=[-Q_5]\times[\overline M_{0,E_v\cup S_v}]$;\footnote{when $|E_v\cup S_v|\le2$, $[\overline M_{0,E_v\cup S_v}]=[pt]$.}
\item
For   a web $\fa$ of $\Theta$ at infinity, the regularity of $\Theta$ implies all legs of $\fa$ are narrow,  and $[\cW_\fa]\virt$ is  defined as
$[\cW_{(\Theta')}]\virt$ for $\Theta':=\fa$. \end{enumerate}
\end{lemm}


\subsection{The moving parts}
We state the virtual localization formula. Let $\iota_{\Theta}:\cW_{(\Theta)}^-\to\cW^-$ be the inclusion.

\begin{theo}[{\cite{GP, CoVir}}]
We have the virtual localization formula,
$$[\cW]\virt=\sum_{\Theta\in G\lggd^\reg} \iota_{\Theta\ast}\Bigl( \frac{[W_{(\Theta)}]\virt}{e(N_\Theta\virt)}\Bigr),
$$
where the equality holds after inverting non-zero positive degree homogeneous elements in $H\sta_G(\pt)$. 
\end{theo}

Recall that $(e,v)\in F$ means that the edge is incident to the vertex $v$; $(e,v)\in F^S$ means that in addition $v$ is
stable.

\begin{prop}\label{final}
The virtual Euler class $e(N_\Theta\virt)$ is given by
$$\frac{1}{e_G(N^\vir_{\vGa})} =\biggl( \prod_{(e, v)\in F^S} A_{(e,v)}\cdot \prod_{v\in V-V_\infty}A_v
\cdot \prod_{e\in E_{01}\cup E_{1\infty}\cup E_{11}} A_e\biggr) \cdot  A_{\infty}.
$$
\end{prop}

We now explain the individual term appearing in the formula. We first introduce
\[V^{a, b}=\{v\in V-V^S\,|\, |S_v|=a, \, |E_v|=b\}, \quad V_\bullet^{a, b}=V_\bullet\cap V^{a, b}.
\]
We set  $ \Pi(\alpha):=\prod\limits_{ {\beta=1,\beta\neq \alpha}}^N({t}_\beta-{t}_\alpha)$.
For $e\in E_{1\infty}$,  set $r_e =1$ when $ d_e \in \ZZ$, and $r_e=5$ otherwise.

\medskip
\noindent
{\sl Formula for $w_{(e,v)}$} (for the definition, see \cite[\S4.1]{CLLL2}).

Let $v$ and $v'$ be the two vertices of an   edge $e\in E_{01}\cup E_{1\infty}$.
Let $v'\in V_1^\alpha$. Let $\delta=-1$ when $v\in V_\infty^{0,1} $, and $\delta=0$ otherwise.
\begin{itemize}
\item
in case $v\in V_0$,  $w_{(e,v)}=\dfrac{h_e +{t}_\alpha}{d_e}$ and $ w_{(e,v')}=-\dfrac{h_e+{t}_\alpha}{d_e}$;
\item
in case $v\in V_\infty^\alpha-V^{0,1}$,
$w_{(e,v)}=\dfrac{{t}_\alpha}{r_ed_e}$ and $w_{(e,v')}=\dfrac{-{t}_\alpha}{ d_e}$;
\item
in case $v\in V_\infty^\alpha\cap V^{0,1}$,
$w_{(e,v)}=\dfrac{5{t}_\alpha}{5d_e+1}$ and $w_{(e,v')}=\dfrac{- 5 {t}_\alpha}{ 5d_e+1 }$;

\item let $e\in E_{11}^{\alpha\beta}$, $v\in V_1^\beta$,
$w_{(e,v)}=\dfrac{{t}_\alpha-{t}_\beta}{d_e}$
\end{itemize}

\medskip
\noindent
{\sl Vertices in $V^{0,2}$}.

Let $v$ be a vertex in $V^{0,2}$; set $E_v=\{e,e'\}$. Then
\begin{itemize}
\item
when $v\in V_0^{0,2}$, $A'_v=\prod\limits_{\alp=1}^{N} (h_e +{t}_\alp) = \prod\limits_{\alp=1}^{N} ( h_{e'}+{t}_\alp )$;
\item
when $ v\in V_1^{0,2}\cap V_1^\alp$, $A'_v=-5{t}_\alp^6 \cdot \Pi(\alpha)$;
\item
when $v\in V_\infty^{0,2}, d_e\notin \ZZ$, $A'_v=\Pi(\alpha)$.  
\end{itemize}
	
\medskip
\noindent
{\sl Flags}.

Let $(e,v)\in F^S$. We have
\begin{itemize}
\item
when $v\in V_0$, $A'_{(e,v)}=\prod\limits_{\alp=1}^{N}(h_e+{t}_\alp)$;
\item
when $v\in V_1^\alp$, $A'_{(e,v)}=-5{t}_\alp^6\cdot \Pi(\alpha)$;
\item
when $v\in V_\infty^\alp$ and $q(e,v)=\sC_e\cap\sC_v\ \text{is stacky}$, $A'_{(e,v)}= \Pi(\alpha)$.
\end{itemize}

\medskip
\noindent
{\sl Edges}.

Let $e$ be an edge with $v$ and $v'$ incident to it. Let $v'\in V_1^\alpha$;
$d_e=\deg \sL|_{C_e}$, $d_{\infty,e}=\deg \sN|_{C_e}$.
\begin{itemize}
\item
when $e\in E_{01}$, 
$
A_e = \frac{\prod\limits_{j=1}^{5d_e-1}(-5 h_e +\frac{j(h_e+{t}_\alp)}{d_e})}{
			\prod\limits_{j=1}^{d_e}(h_e-\frac{j(h_e+{t}_\alp)}{d_e})^5 \prod\limits_{j=1}^{d_e}\frac{j(h_e+{t}_\alp)}{d_e}  \prod\limits_{\substack{ \beta=1\\\beta\neq \alpha}}^N \prod\limits_{j=0}^{d_e}
			\left(\frac{j(h_e+{t}_\alpha)}{d_e} +{t}_\beta-{t}_\alpha\right)} $;
\item
when $e\in E_{1\infty}$,  
$A_e=\frac{\prod\limits_{j=1}^{\lceil-d_e\rceil-1}(-{t}_\alp-\frac{j{t}_\alp}{d_e-\delta/5} )^5 }{
			\prod\limits_{j=1}^{{ -5d_e+\delta}}(\frac{-j{t}_\alp}{d_e-\delta/5})
			\prod\limits_{j=1}^{\lfloor d_{\infty,e} \rfloor}(\frac{j{t}_\alpha}{-d_{\infty,e}-\delta/5 })}\cdot \frac{1}{\Pi(\alpha)}$;
\item
when $e\in E_{11}^{\alpha\beta}$, incident to vertices $v\in V_1^\alpha$ and $v'\in V_1^\beta$ ($\alpha\ne\beta$), \
assuming $\Sigma^{\sC}_{(1,\varphi)}=\emptyset$,

$A_e = (-1)^{d_e} \frac{(d_e)^{2d_e} }{(d_e!)^2 (t_\beta-t_\alp)^{2d_e} } \cdot
\frac{\prod\limits_{i=1+\delta^\prime+\delta_\rho^\prime}^{5d_e-1 -\delta-\delta_\rho}
\big(5t_\alpha-\frac{i(t_\alpha-t_\beta)}{d_e}\big) }{\prod\limits_{a+b=d_e} \bl ( -\frac{a}{d_e}t_\alp - \frac{b}{d_e}t_\beta )^{5}
\prod\limits_{\gamma\neq \alp,\beta} ( t_{\gamma}-\frac{a}{d_e}t_\alp - \frac{b}{d_e}t_\beta ) \br}$.
\end{itemize}
Here, $\delta'=-1$ when $v'\in V_1^{0,1}$, and $\delta'=0$ otherwise; $\delta_\rho'=-1$ when $v'\in V_1^{1,1}$ and
$q_{(e,v')}\in \Sigma_{(1,\rho)}^\sC$.

\medskip
\noindent
{\sl Stable vertices}.

Let $v\in V^S$, let $\pi_v: \cC_v\to \cW_v$ (defined in \cite{CLLL2}) be the universal curve;
let $\cL_v$ be the universal line bundle over $\cC_v$, and let $\mathbb E_v\colon=\pi_{v*}\omega_{\pi_v}$ be the
Hodge bundle, where $\omega_{\pi_v}$ is the relative dualizing sheaf over $\sC_v$.
\begin{itemize}
\item when $v\in V^S_0$,
$A'_v=\frac{1}{\prod\limits_{\beta=1}^N e_G(R\pi_{v*} \phi_v^* \cO_{\PP^4}(1)\otimes \bL_{\beta} )}$;
\item when $v\in V_1^S\cap V_1^\alpha$, $A'_v=
\Big(\frac{e_G(\EE_v^\vee \otimes \bL_{-\alpha})}{-{t}_\alpha}\Big)^5
	\cdot \frac{(\frac{-{t}_\alpha^4}{5})^{|S_v^{(1,\varphi)}|}\cdot\prod\limits_{\beta=1, \beta\neq \alpha}^Ne_G(\EE_v^\vee\otimes \bL_{{\beta}-{\alpha}  })}{e_G(\EE_v\otimes \bL_{5\alpha})\cdot({5{t}_\alpha})^{|E_v|-1}\cdot\Pi(\alpha)}$;
\item when $v\in V_\infty^S\cap V^\alpha_\infty$, $A'_v=
\frac{1}{e_G(R\pi_{v*}\cL_v^\vee \otimes \bL_{-\alpha})} \cdot
\frac{\prod\limits_{\beta=1, \beta\neq \alpha}^Ne_G(\EE_v^\vee\otimes \bL_{{\beta}-{\alpha}  })}{\Pi(\alpha)
}$.
\end{itemize}

\medskip
\noindent
{\sl Unstable vertices}.

\begin{itemize}
\item when $v\in V_1^{0,1}\cap V^\alp_1$, $A'_v=5{t}_\alp$;
\item when $v\in V_1^{1,1}\cap V^\alp_1$ and $S_v\subset \Si^{\sC}\lophi$, $A'_v=-t^5_\alp$;
\item when $v\in V_1^{1,1}\cap V_1^\alp$ and $S_v\subset \Si^{\sC}\lorho$, $A'_v=5{t}_\alp$.
\end{itemize}

\medskip
\noindent
{\sl The terms $A_a$'s}.

\begin{itemize}
\item when $(e,v)\in F^S$, $A_{(e,v)}=\frac{A'_{(e,v)} } {w_{(e,v)}-\psi_{(e,v)}}$;
\item when $v\in V^{0,2}$ and $E_v=\{e,e'\}$, $A_v=\frac{A'_v}{w_{(e,v)}+w_{ (e',v)\black}}$;
\item when $v\in V^{0,1}$ and $E_v=\{e\}$, $A_v=A'_v\cdot w_{(e,v)}$;
\item when $v\in V^S\cup V^{1,1}$, $A_v=A'_v$;
\item when $v\in V_0^{0,1}\cup V_0^{1,1}\cup V_\infty^{0,1}\cup V_\infty^{1,1}$,  $A'_v=1$.
\end{itemize}

\subsubsection{The term $A_{\infty}$}
We now compute the term $A_{\infty}$.
%
%
Firstly, the virtual normal bundle $N_\Theta\virt$  comes from two parts:
one is from deforming the curve $(\Sigma^\sC\sub\sC)$, and the other is from
deforming $(\sL,\sN,\varphi,\rho,\mu,\nu)$.

For the part in deforming $\Sigma^\sC\sub\sC$, they have been accounted for except when deforming the part
of $\Sigma^\sC\sub\sC$ contained   entirely in $\sC_\infty$, namely the part given by the webs $\fa_1,\cdots,\fa_s$.

For the second part, we let $\cV=\sL(-\Sigma_{(1,\varphi)})^{\oplus 5}\oplus
\cdots$ be the sheaves where the sections $(\varphi,\rho,\mu,\nu)$ lie. Then using the standard long exact sequence
of cohomologies, we obtain an identity
in K-theory:
\begin{align}\nonumber
H^0(\cV-\sO_\sC^{\oplus 2})-H^1(\cV-\sO_\sC^{\oplus 2})=\bigoplus _{a\in V^S\cup E} &\left(H^0(\cV|_{\sC_a}-\sO_{\sC_a}^{\oplus 2})-H^1(\cV|_{\sC_a}-\sO_{\sC_a}^{\oplus 2})\right)-\\
&- \bigoplus_{a\in F^S\cup V^{0,2}} H^i(\cV|_{q_a}-\sO_{q_a}^{\oplus 2}).\label{KK}
\end{align}\black
For the first group of terms on the RHS of \eqref{KK}, all but $a\in V^S_\infty\cup E_{\infty\infty}$ are accounted for
in the identity in Proposition \ref{final}; for the second group of terms on the RHS of \eqref{KK} all but
$a\in  \cup_{i=1}^s\bl F^S_{\fa_i}\cup V^{0,2}_{\fa_i}\br$ are accounted for.
Thus $A_\infty$ is the combined contributions from the unaccounted for terms.
In conclusion,
$$A_\infty=\frac{1}{\prod_{i=1}^s e_G(N_{\fa_i}\virt)},
$$
where each $N_{\fa_i}\virt$
is the virtual normal bundle of $\cW_{\fa_i}$ in the respective moduli of NMSP fields.

\subsection{Separation of nodes}
For our application, we will work with the case where $\gamma$ consists of $n$ many $(1,\varphi)$-type markings,
and $\bd=(d,0)$. Thus the set of its associated regular graphs are $G_{g,n,d}\ureg$, etc..

Let $\Theta$ be a regular graph in $G_{g,n,d}\ureg$; let $\Cont_\Theta$ be its associated cycle in the localization formula of $[\cW]\virt$
in { \eqref {localization0}}.
In the subsequent paper, we need to decompose the cycle $\Cont_\Theta$ along the nodes in $V_1$. In this  subsection,
we show how such decomposition is derived.

Let $e\in E_{1\infty}$ be an edge of $\Theta$, incident to a vertex  $v\in V_1$.
Let $\xi=(\sC,\Sigma^\sC,\cdots)\in \cW_\ga$.   We assume $q_{(e,v)}=\sC_e\cap \sC_v$  is a node of $\sC$;
when $v$ is unstable we call $e'$ the other edge incident to $v$ if $|E_v|=2$ and $E_v=\{e,e'\}$.


\begin{definition}
We say that the node $q_{(e,v)}$ is a level separating node if either $v$ is stable, or when $v$ is unstable and its
other edge $e'$ (incident to $v$) lies in $E_{11}\cup E_{01}$.
\end{definition}

We now construct a new graph $\Theta'$ that separates at a level separating node $q_{(e,v)}$.
{ The $\Theta'$ is derived from $\Theta$ by making $e$ not incident to $v$,
adding a new level $1$ unstable vertex $v'$ incident to $e$, adding a new $(1,\rho)$-leg $l'$ to $v'$,
and adding a new $(1,\rho)$-leg $l$ to $v$.
This way, $\Theta'$ has a new vertex $v'$ and two new legs $l$ and $l'$.
In case $\Theta$ is connected, then either $\Theta'$ has two connected components or $g_\Theta=g_{\Theta'}+1$.}


Let $v\in V_1^\alpha$ and  $a_e=-5d_{e}$. In case $v$ is stable, let $\psi_{(e,v)}$ be the psi class of $\sC_v$ at $q_{(e,v)}$;
in case $v$ is unstable, then let {  $\psi_{(e,v)}$   be $-w_{(e',v)}$}, where $w_{(e',v)}=e_G(T_{q_{(e,v)}}\sC_{e'})$.
We also abbreviate $\delta(\Theta)=d_\Theta+g_\Theta+1$.

Denote $1_\alpha:=1$ to be the generator of the
$H^0(\pt_\alp)$ where $\pi_\alp$ is the $\alpha$-th $G$-fixed point in $\PP^{4+N}$. Also denote $1^\alpha:=-\frac{1}{5}Nt\lalp^3 t^N\cdot 1\lalp$.

\begin{proposition} Suppose $\Theta'$ is connected, then
$$(-1)^{\delta(\Theta)}\Cont_\Theta(-)=(-1)^{\delta(\Theta')}\Cont_{\Theta'}\bl-, \frac{1^{\alpha}}{\frac{5t\lalp}{a_e }-{  \psi_{(e,v)}}}\br.
$$
Suppose $\Theta'=\Theta_1\coprod \Theta_2$ is disconnected with $v\in V(\Theta_1)$. Then
$$(-1)^{\delta(\Theta)}\Cont_\Theta(-)=(-1)^{\delta(\Theta_1)}\Cont_{\Theta_1}\bl-, \frac{1^{\alpha}}{\frac{  5t\lalp}{a_e}-{\psi_{(e,v)}}}\br\cdot(-1)^{\delta(\Theta_2)} \Cont_{\Theta_2}\bl-\br.
$$
\end{proposition}

\begin{proof}We first look at the case where $v$ is stable.
By looking at the localization formula of $[\cW]\virt$ stated in this section,
we see that the factors in $\Cont_\Theta$ that are related to both
$v$ and $e$ are
\beq\label{fac}
\bl\frac{1}{5t\lalp}\br^{|E_v|}\ \text{in} \ A_{v}',\quad A_{v'}', \and \frac{A_{(e,v)}'}{w_{(e,v)}-\psi_{(e,v)}}.
\eeq
({See arguments two paragraphs below.}) After splitting, $\Theta'$ will have a new vertex $v'$ incident to $e$,
one less edge incident to $v$, one additional $(1,\rho)$-leg incident to $v'$, and one
additional $(1,\rho)$-leg incident to $v$. This will contribute an extra $\frac{1}{5t\lalp}$.
On top of this, since $e$ is no longer incident to $v$, the term $A_{(e,v)}'$ is missing. { Using that
$$A_{(e,v)}'={-5t\lalp^6\cdot \prod\limits_{\beta\ne \alpha}(t_\beta-t_\alpha)}={5Nt\lalp^5 t^N},\and
1^\alpha=-\frac{1}{5}Nt\lalp^3 t^N\cdot 1\lalp,
$$
we obtain}
$$\frac{\text{Factors in $\Cont_\Theta$}}{\text{Factors in $\Cont_{\Theta'}$}}=\frac{1}{5t\lalp}\cdot
\frac{1}{5t\lalp}\cdot \frac{5Nt\lalp^5 t^N}{w_{(e,v)}-\psi_{(e,v)}}\cdot 1\lalp=\frac{\frac{1}{5}Nt\lalp^3 t^N\cdot
1_\alpha}{w_{(e,v)}-\psi_{(e,v)}}
=-\frac{1^\alpha}{\frac{a_et\lalp}{5}-\psi_{(e,v)}}.
$$
Note that we have a minus sign in front of the relevant term.

We now assume $\Theta$ is connected.
When $\Theta'$ is connected, then $\delta(\Theta)=\delta(\Theta')+1$;
when $\Theta'=\Theta'_1\coprod \Theta_2'$ is disconnected, we have
$\delta(\Theta)=\delta(\Theta_1)+\delta(\Theta_2) -1$.
This extra $1$ cancels the minus sign mentioned before. This proves the Proposition when $v$ is stable.
The case where $v$ is unstable is similar, and will be omitted.
\end{proof}

\begin{proof}[Proof of Theorem \ref{thm2}]
Let $\bd=(d,d_\infty)$. Recall that
$$[\cW\gnD]\virt\in A_{\delta(\bd)}^G \bl \cW\gnD^-\br,
$$
where $\cW\gnD^-$ is the degeneracy locus of the cosection of the obstruction sheaf, which is proper
and $G$-invariant, and (cf. \eqref{vdim})
$$\delta(\bd)=Nd+d_\infty+N(1-g)+n.
$$
Let $\pi: \cW\gnD^-\to \pt$ be the proper projection. By definition, $\bigl<\prod_i\tau_i\bar\psi_i^{k_i}\bigr>^M_{g,n,d_\infty}$
is a power series in $q$, with the coefficient of $q^d$ being
\beq\label{coef}
\mathbf c_d:=(-1)^{d+1-g} \pi\lsta\bigl( \prod_i\tau_i\bar\psi_i^{k_i}\cdot \bigl[\cW_{g,n,(d,d_\infty)}\bigr]\virt\bigr)\in
A^G_{\delta(\bd)-\sum^n \deg\tau_i-|k_\bullet|}(\pt).
\eeq
Here $|k_\bullet|=\sum^n k_i$.
Therefore, $\mathbf c_d$ can be possibly nonzero only when
$$\delta(\bd)-\sum^n (\deg\tau_i +k_i )= N(d+1-g-\eps)\le 0,
$$
(recall $\eps$ is defined in Theorem \ref{thm2}), which  is equivalent to
$d\le (g-1)+\eps$.

This shows that
\beq\label{poly2}\bigl<\prod_i\tau_i\bar\psi_i^{k_i}\bigr>^M_{g,n,d_\infty}=
\sum_{0\le d\le (g-1)+\eps} \mathbf c_d q^d,
\eeq
is a polynomial of degree bounded by $g-1+\eps$.

Following our convention, after equivariant integration we will substitute $t\lalp$ by $-\zeta_N^\alpha t$,
thus \eqref{poly2} is indeed a polynomial in $q$ with coefficient a Laurent monomial in $t$.
As $\mathbf c_d$ has degree $\delta(\bd)-\eps-|k_i|$, $\mathbf c_d=c_d\cdot t^{-\delta(\bd)+\eps+|k_i|}$, $c_d\in \QQ(\zeta_N)$.
Substituting the expression of $\delta(\bd)$, we prove Theorem \ref{thm2}.
\end{proof}

\end{document}